\RequirePackage{fix-cm}
\documentclass[smallextended]{svjour3}       
\smartqed  
\usepackage{graphicx}
\usepackage[utf8]{inputenc}
\usepackage[dvipsnames]{xcolor}
\usepackage{mathtools}
\usepackage{hyperref}
\usepackage[justification=centering]{caption}
\usepackage{subcaption}
\DeclarePairedDelimiterX{\inp}[2]{\langle}{\rangle}{#1, #2} 
\usepackage[export]{adjustbox}
\usepackage{longtable}
\usepackage{physics} 
\usepackage{amssymb} 

\usepackage{amsmath,bm}
\usepackage{acronym}
\usepackage{algorithm,algpseudocode}
\algnewcommand{\Inputs}[1]{%
  \State \textbf{Inputs:}
  \Statex \hspace*{\algorithmicindent}\parbox[t]{.8\linewidth}{\raggedright #1}
}
\algnewcommand{\Outputs}[1]{%
  \State \textbf{Outputs:}
  \Statex \hspace*{\algorithmicindent}\parbox[t]{.8\linewidth}{\raggedright #1}
}
\algnewcommand{\Initialize}[1]{%
  \State \textbf{Initialize:}
  \Statex \hspace*{\algorithmicindent}\parbox[t]{.8\linewidth}{\raggedright #1}
}
\acrodef{ADMM}[ADMM]{\emph{Alternating Direction Method of Multipliers}}

\acrodef{DMD}[DMD]{\emph{Dynamic Mode Decomposition}}
\acrodef{RDMD}[RDMD]{\emph{Rescaled Dynamic Mode Decomposition}}
\acrodef{FGD}[FGD]{\emph{Fastest Gradient Descent}}
\acrodef{TV}[TV]{\emph{Total Variation}}
\acrodef{BV}[BV]{\emph{Bounded Variation}}
\acrodef{SVD}[SVD]{\emph{Singular Vector Decomposition}}
\acrodef{KMD}[KMD]{\emph{Koopman Mode Decomposition}}
\acrodef{KEF}[KEF]{\emph{Koopman Eigenfunction}}

\acrodef{PDE}[PDE]{\emph{Partial Differential Equation}}

\newcommand{\off}[1]{{\color{yellow}{[removed]}}}

\begin{document}

\title{Total-Variation - Fast Gradient Flow and Relations to Koopman Theory \thanks{We acknowledge suuport by grant agreement No. 777826 (NoMADS), by the Israel Science Foundation (Grant No.  534/19) and by the Ollendorff Minerva Center.}
}


\author{Ido Cohen         \and Tom Berkov \and 
        Guy Gilboa}


\institute{I. Cohen \at
              Tel.: +927-4-829-5743\\
              \email{idoc@campus.technion.ac.il} 
           \and
           T. Berkov \at
              \email{ptom@campus.technion.ac.il} 
              \and
           G. Gilboa \at
              \email{guy.gilboa@ee.technion.ac.il}
}

\date{\today}

\maketitle

\begin{abstract}
The space-discrete \ac{TV} flow is analyzed using several mode decomposition techniques.
In the one dimensional case, we provide analytic formulations to \ac{DMD} and to \ac{KMD} of the \ac{TV} flow and compare the obtained modes to \ac{TV} spectral decomposition. 
We propose a computationally efficient algorithm to evolve the one-dimensional \ac{TV} flow. A significant speedup by three orders of magnitude is obtained, compared to iterative minimizations.  
A common theme, for both mode analysis and fast algorithm, is the significance of phase transitions during the flow, in which the subgradient changes. 

We explain why applying \ac{DMD} directly on \ac{TV}-flow measurements cannot model the flow or extract modes well. We formulate 
a more general method for mode decomposition that coincides with the modes of \ac{KMD}. This method is based on the linear decay profile, typical to \ac{TV}-flow. 
These concepts are demonstrated through experiments, where additional extensions to the two-dimensional case are given. 

\keywords{Isotropic and Anisotropic \acl{TV}\and \acl{TV}-flow  \and \acl{TV}-spectral decomposition \and \acl{DMD} \and Time Reparametrization \and \acl{KMD}}
\end{abstract}
\newpage
{\bf{List of abbreviations}}
\addcontentsline{toc}{chapter}{List of Symbols}
\begin{longtable}{lp{0.6\textwidth}}
\acs{ADMM}&\acl{ADMM}\\
\acs{DMD}&\acl{DMD}\\
\acs{FGD}&\acl{FGD}\\
\acs{KEF}&\acl{KEF}\\
\acs{KMD}&\acl{KMD}\\
\acs{PDE}&\acl{PDE}\\
\acs{RDMD}&\acl{RDMD}\\
\acs{SVD}&\acl{SVD}\\
\acs{TV}&\acl{TV}\\
  \end{longtable}
  
{\bf{List of variables}}
\addcontentsline{toc}{chapter}{List of Symbols}
\begin{longtable}{lp{0.7\textwidth}}
$P(\cdot)$& a proper nonlinear operator\\
$\psi(t)$&The solution of a nonlinear PDE\\
$t$&The time variable\\
$t(\psi)$& The inverse function from the observations $\psi$ back to the time $t$\\
$\varphi_{\mathcal{K}}(\psi)$&\acf{KEF}\\
$\bm{v}_i$&A Koopman mode\\
$a_{\lambda_i}(t)$ & The decay profile typical to the dynamical system\\

$J_{TV}$&Total variation\\
$J_{iso}$&Isotropic total variation\\
$J_{ani}$&Anisotropic total variation\\
$T_i$&The time point at which the subgradient of TV-flow changes for the $i$th time.\\
$p$& a negative subgradient\\
$p_i$& The negative subgradient of the TV-flow for $t\in [T_i,T_{i+1})$\\
$p^i$&The negative subgradient in the $i$th extremum of the signal\\
$p^{(i)}$&The value of the $i$th pixel in the negative subgradient $p$\\

$\tau$&The reparametrized time variable \\
$\varphi_i$&\ac{TV} spectral component\\
$\xi_1,\xi_2$&Modes resulting from the rescaled \ac{TV}-flow\\
$\phi_1,\phi_2$&Modes resulting from \ac{DMD} applied on the rescaled \ac{TV}-flow\\
  \end{longtable}  
\section{Introduction}
Dimensionality reduction is one of the most challenging tasks in data analysis. For flows, the general aim is to find a low dimensional representation of the spatial data (modes) and the corresponding temporal evolution \cite{zhang2020evaluating}. 
This allows much better understanding and facilitates fast and efficient processing of the signal. 
The spatio-temporal decomposition is commonly used in many disciplines such as fluid dynamics, signal analysis and computational physics, to name a few. In this work, we bridge between the \ac{TV} spectral decomposition from signal processing and its connection to Koopman theory and its applications. 

In order to study nonlinear dynamics,  Koopman theory \cite{koopman1931hamiltonian} is  increasingly employed. It allows a linear representation of nonlinear flows , also referred to as  \ac{KMD} \cite{mezic2005spectral}. Since in the general case the representation is infinite-dimensional, in practice, a finite dimensional approximation is used.  In fluid dynamics, \acf{DMD} is a data-driven method to approximate \ac{KMD} \cite{schmid2010dynamic}. Since \ac{DMD} approximates a linear dynamical system it can be viewed as an exponential data fitting algorithm \cite{askham2018variable}. \ac{DMD} has become a popular and practical analysis tool, however, it  suffers from some inherent flaws \cite{cohen2020introducing,cohen2021examining,rosenfeld2021singular,elmore2021learning}. These drawbacks are emphasized in zero-homogeneous evolutions such as the \ac{TV} flow.

The \ac{TV} functional has been extensively used in image processing and computer vision, as it allows the preservation of edges (discontinuities). The respective steepest descent, termed \ac{TV}-flow, has been studied thoroughly in the last two decades \cite{tvFlowAndrea2001,steidl2004equivalence,steidl2005relations,brox2003equivalence,chambolle2004algorithm,bonforte2012total,tvFlowAndrea2001,bellettini2002total}. A list of the main attributes of \ac{TV} flow in a one-dimensional signal is summed up in \cite{brox2006tv}. 
Two of its main attributes are that the flow decays piecewise linearly and that it has a finite support in time. The above studies lay the foundation to the \ac{TV} nonlinear spectral framework (spectral \ac{TV}) \cite{gilboa2013spectral,gilboa2014total,burger2016spectral,gilboa2016nonlinear,bungert2021nonlinear,brokman2021nonlinear,fumero2020nonlinear}. A main challenge in accurately using \ac{TV} flow is the computational load in calculating the subgradient. For an accurate calculation, often a convex optimization problem  (such as  \cite{chambolle2004algorithm}) is solved at each step. In order to obtain a full spectral \ac{TV} decomposition of a signal, one has to compute the entire \ac{TV} flow until extinction. This motivates us to explore fast methods to accomplish the task.

In this work, we apply the theory of Koopman operator and its applications on the \ac{TV}-flow. We begin by formulating the analytic solution of \ac{TV}-flow for one-dimensional signals and propose a fast algorithm to solve the flow. 
We first examine \ac{RDMD} \cite{cohen2020modes} which was recently proposed as an improved version of \ac{DMD} for homogeneous flows (such as $p$-Laplacian flows \cite{kuijper2007p}, 
where the operator of the dynamics is  homogeneous of degree $p-1$). However, since there are subgradient phase transitions in \ac{TV}-flow, \ac{RDMD} is limited. To address this problem we suggest a new piecewise mode decomposition technique. It is based on the linear decay profile typical to the flow. We show it coincides with \ac{KMD} and can perfectly model the dynamics through piecewise linear segments. We show the direct relations of the obtained modes to spectral \ac{TV} decomposition \cite{gilboa2014total}. 
A shorter version of this work was first presented in a conference  \cite{cohen2021Total}. 
Here, additional theory and experiments are shown, and a two-dimensional extension is presented.


\section{Preliminaries}
In this section, we summarize the definitions and methods relevant to this work.

\subsection{\aclp{KEF}}
Let us consider the following nonlinear dynamical system
\begin{equation}
	\frac{d}{dt}\psi=P(\psi),
\end{equation}
where $P(\cdot)$ is a proper nonlinear operator. The observation $\psi$ belongs to $\mathbb{R}^M$. We also denote the time derivative with the subscript $\{\cdot\}_t$ for simplicity.

Let $g:\mathbb{R}^M \to \mathbb{R}$ be a measurement of $\psi$. Bernard Osgood Koopman argued that a Hamiltonian system has measurements evolving linearly under the dynamical regime \cite{koopman1931hamiltonian}. These measurements, termed as \acfp{KEF}, admit the following relation, 
\begin{equation}
\frac{d}{dt}	\varphi_{\mathcal{K}}(\psi(t))=\lambda \varphi_{\mathcal{K}}(\psi(t)),
\end{equation}
where $\varphi(\psi)$ is a \ac{KEF} and $\lambda$ is the respective eigenvalue. 
Recently, part of the authors formulated necessary and sufficient conditions for the existence of such measurements \cite{cohen2021examining}. It was shown that if the solution of the system, $\psi(t)$, is injective then the measurement $\varphi$ gets the form of,
\begin{equation}\label{eq:KEF}
	\varphi_{\mathcal{K}}(\psi)=\varphi_{\mathcal{K}}(\psi(0))e^{\lambda t(\psi)},
\end{equation}
where $t(\psi)$ is the inverse function from the observations $\psi$ back to the time $t$. For further details see \cite{cohen2021examining}.

\subsection{\acl{KMD}}
\acf{KMD} is a representation of dynamical systems based on \acp{KEF} \cite{mezic2005spectral}. Namely, the state space $\psi$ can be expressed as,
\begin{equation}
	\psi(t)=\sum_{i=1}^\infty\bm{v}_i\varphi_{\mathcal{K}_i}(t),
\end{equation}
where $\varphi_i(t)$ is a \ac{KEF} and $\bm{v}_i$ is the corresponding vector, referred to as Koopman mode. When the dynamic is nonlinear the decomposition may be infinite. In practice, a finite approximation method is used. The most common one is \ac{DMD}.

\subsection{\acl{DMD}}
\acf{DMD} is a data-driven algorithm to approximate \ac{KMD}. Given a set of samples of the dynamic, a linear relation is formulated that approximates  the generating process of the set. Commonly, the formulation is performed in a lower dimensional space. The output is triplets representing the main spatial structures (modes), their amplitudes (coefficients), and the respective time changes (eigenvalues). To recap, the three main steps of this algorithm are: 1. \emph{Dimensionality reduction}, 2.  \emph{Optimal linear mapping}, and 3. \emph{System reconstruction} . The output is \emph{modes, eigenvalues and coefficients} \cite{schmid2010dynamic}.

\subsection{Decay profile profile decomposition with Koopman modes}\label{subsec:DecayProfileKMD}
We summarize below a different decomposition method, suggested in \cite{cohen2021examining}, which is relevant to this work.
Let us assume that the dynamical system can be formulated as, 
\begin{equation}\label{eq:DPMD}
\psi(t)=\sum_{i=1}^{L}\bm{v}_ia_{\lambda_i}(t),
\end{equation}
where $a_{\lambda_i}(t)$ is the decay profile typical to the dynamical system and characterized by a parameter (or set of parameters) $\lambda_i$ and $\bm{v}_i$ is a spatial structure. In matrix notations,
\begin{equation}
\psi(t)= V\bm{{a}_{\lambda}}(t),
\end{equation}
where $\bm{a}_{\bm{\lambda}}(t)=\begin{bmatrix}a_{\lambda_1}(t)&\cdots&a_{\lambda_L}(t)\end{bmatrix}^T$ and $V$ is a matrix with the modes as its columns.
If the decay profile is monotone, the inverse mapping exists and can be expressed as,
\begin{equation}
\bm{t}(\psi) = \bm{a}_{\bm{\lambda}}^{-1}\left((V^TV)^{-1}V^T\psi\right),
\end{equation}
where $\bm{a}_{\bm{\lambda}}^{-1}(\cdot)=\begin{bmatrix}a_{\lambda_1}^{-1}(\cdot)&\cdots&a_{\lambda_L}^{-1}(\cdot)\end{bmatrix}^T$. Note that $\bm{t}$ is a vector of inverse mapping. Therefore, the Koopman eigenfunctions are,
\begin{equation}
\bm{\varphi_{\mathcal{K}}}(\psi) = \exp\{\bm{a}_{\bm{\lambda}}^{-1}\left((V^TV)^{-1}V^T\psi\right)\},
\end{equation}
where $\bm{a}_{\bm{\lambda}}^{-1}(\cdot)=\begin{bmatrix}a_{\lambda_1}^{-1}(\cdot)&\cdots&a_{\lambda_L}^{-1}(\cdot)\end{bmatrix}^T$. 

The general mode decomposition algorithm is data-driven. We mention here only the part in the algorithm revealing the spatial structures. For the rest of the algorithm, we refer the reader to \cite{cohen2021examining}. Given the data matrix, $\Psi$, we initialize the decay profile dictionary, $\mathcal{D}$. The columns of the matrix  are samples of the decay vector $\bm{a_\lambda}(t)$. We minimize the expression,
\begin{equation}
\norm{\Psi-VD}^2_F
\end{equation}
over $V$, where $V$ is sparse columns-wise (the substrict $F$ denotes the Frobenius norm) . Then, we construct the matrices $\hat{V}$ and $\hat{D}$. $\hat{V}$ contains the non-zero (or non-negligible) modes (vectors in $V$) and $\hat{D}$ contains the respective decay profiles. 
Note that, the mode set $\{\bm{v}_i\}$ is proved to be the set of Koopman modes \cite{cohen2021examining}.

\subsection{\acl{TV} Spectral Decomposition}\label{sec:tv_intro}



\subsubsection{\ac{TV} functional}
\paragraph{One dimensional signal.}
The \ac{TV} functional of a one dimensional function is defined by,

\begin{equation}
    J_{TV}(\psi)=\inp{\abs{\nabla \psi}}{1},\quad \psi\in \mathbb{R}^M,
\end{equation}
where $\nabla$ is the discrete gradient operator. (for more details we refer the reader to \cite{chambolle2010introduction}).
\paragraph{Two dimensional signal.} For a two dimensional function, the definitions of isotropic and anisotropic \ac{TV}, respectively, are,
\begin{equation}
\begin{split}
    J_{iso}(\psi)&=\inp{\norm{\nabla \psi}_2}{1}=\inp{\sqrt[]{\abs{D_x \psi}^2+\abs{D_y \psi}^2}}{1},\quad \psi\in\mathbb{R}^{M\times K}\\
    J_{ani}(\psi)&=\inp{\norm{\nabla\psi}_1}{1}=\inp{\abs{D_x \psi}+\abs{D_y \psi}}{1},\quad \psi\in\mathbb{R}^{M\times K},
\end{split}
\end{equation}
where $\nabla = [D_x, D_y]^T$ is the discrete gradient operator and $D_x,D_y$ denote the differential operators according to the Cartesian coordinate system
(see e.g. \cite{lou2015weighted} for a new model combining these two functionals).

\subsubsection{\ac{TV}-flow}
The \ac{TV}-flow is the gradient descent flow of the \ac{TV} functional, 
\begin{equation}\label{eq:ss} \tag{{\bf{TV-flow}}}
\begin{split}
    \psi_t&=p,\quad \psi(t=0)=f,
\end{split}
\end{equation}
where $-p$ belongs to the subdifferential, $-p \in \partial J_{TV}(\psi)$, defined by,
\begin{equation}
    \partial J_{TV}(\psi)=\{p | J(\theta) - J(\psi)\ge \inp{-p}{\theta-\psi}, \forall \theta\in \mathcal{H}\},
\end{equation}
and $\mathcal{H}$ is either $\mathbb{R}^M$ or $\mathbb{R}^{M\times K}$, depending on the dimensionality of $\theta$.
A nonlinear eigenfunction, $v$, of $p$ admits,
\begin{equation}\label{eq:EF}\tag{\bf{EF}}
    p(v)=\lambda\cdot v,
\end{equation}
for some non-positive $\lambda\in \mathbb{R}^-$.
The solution of Eq. \eqref{eq:ss} initialized with an eigenfunction $v$ is,
\begin{equation}
    \psi(t) = \left(1+\lambda t\right)^+\cdot v,
\end{equation}
where $\lambda$ is the corresponding eigenvalue and $(a)^+:=\max\{a,0\},\, \forall a\in \mathbb{R}$.

\subsubsection{\ac{TV} spectral framework}
The spectral decomposition of a signal, $f\in \mathcal{H}$, related to the eigenfunctions of $P$ is based on the solution of Eq. \eqref{eq:ss}. The definition of the \ac{TV} transform is given by \cite{gilboa2014total},

\begin{equation}\label{eq:oneTrans}
    \mathcal{G}(t) = t \frac{d^2}{dt^2}\psi(t), 
\end{equation}
where $\psi(t)$ is the solution of \eqref{eq:ss}. The function $\mathcal{G}(t)$ is the spectral component of the signal $f$ at time $t$.
For example, the transform of an eigenfunction is,
\begin{equation}\label{eq:TVtransformEF}
\begin{split}
    \mathcal{G}(t) = f\cdot t\lambda^2\cdot\delta(1+\lambda\cdot t),
\end{split}
\end{equation}
where $\delta(\cdot)$ is the Dirac measure.
We list below some \emph{Attributes} of the semi discrete one-dimensional TV-flow and the TV spectral components:
\begin{enumerate}
    \item The subgradient is piecewise constant with respect to $t$ (see e.g. \cite{burger2016spectral}). \label{Att:piecewise}
    \item The initial condition can be reconstructed by knowing the subgradient as a function of $t$ (by integration). \label{Att:recons}
    \item The flow splits into merging events \cite{brox2006tv}. \label{Att:merging}
    \item The average of a subgradient over the spatial variable is zero.\label{Att:Average}
    \item The spectrum is a finite set of delta functions, where each delta function represents a spectral component \cite{burger2016spectral}.\label{Att:FiniteDelta}
    \item For a given $f$, the spectral component set is orthogonal \cite{burger2016spectral}.\label{Att:orthogonal}
    \item Two adjacent points which become equal in value during the flow, will not separate \cite{steidl2004equivalence,bellettini2002total}. \label{Att:adjacent}
\end{enumerate}

{\bf{Settings:}} In this work we first note that \ac{DMD} is fully discrete (time and space) whereas \ac{TV}-flow and spectral \ac{TV} are semi-discrete (time-continuous, spatially discrete). Thus, in order to apply \ac{DMD} on a gradient descent flow we first need to sample (uniformly) with respect to the time variable $t$. In all cases we use Euclidean inner product and norm.

\section{One dimensional \texorpdfstring{\ac{TV}}{TEXT} flow, \texorpdfstring{\ac{DMD}}{TEXT}, Koopman eigenfunctions and modes }\label{sec:DMDTV}
\subsection{Optimal calculation of one dimensional \texorpdfstring{\ac{TV}-flow }{TEXT}}

Let us formulate \emph{Attribute \ref{Att:piecewise}} and \emph{Attribute \ref{Att:recons}} more formally.
The solution of \eqref{eq:ss} converges to a steady state in finite time. In this finite time, the solution is divided into $L$ disjoint segments, $\{[T_i,T_{i+1})\}_{i=0}^{L-1}$. In each segment, the subgradient is constant, 
\begin{equation}
-p_i\in \partial J(\psi(t)), t\in[T_i,T_{i+1}),
\end{equation}
where for $t>T_L$ it is zero, $p_{L+1}=0$. The solution can be expressed by (e.g. \cite{burger2016spectral}),
\begin{equation}\label{eq:GilboaSolu}
    \psi(t) = \psi(T_i)+(t-T_i)p_i,\quad t\in[T_i,T_{i+1}).
\end{equation}
We propose here a fast algorithm to find a subgradient of the \ac{TV} functional of a one-dimensional signal. The results of this algorithm coincide with the subgradient calculated in \cite{steidl2004equivalence}. 

\subsubsection{Calculating a subgradient}
While there has been ongoing research on fast methods for TV regularization (e.g. \cite{cherkaoui2020fast_tv_nips,darbon2006image,goldfarb2009parametric}), few advances were made in fast algorithms of the TV-flow, which is required for computing spectral TV. 
Our proposed solution, $\psi(t)\in \mathbb{R}^M \times [0,T_L]$, is in a semi-discrete setting. The algorithm is based on the \ac{TV}-flow attributes listed at the end of Section \ref{sec:tv_intro}. The \acf{FGD} flow directly stems from the works \cite{brox2006tv,brox2003equivalence,steidl2004equivalence}. We assume here Neumann boundary condition (generalization to other boundary conditions is possible). 

We consider an expression for the value of \ac{TV} of a piecewise monotone signal $f$. Let the set $\{f_i\}_{i=1}^N$ be the local extremum points of the signal and let $\{m_i\}_{i=1}^N$ be the number of pixels at every local extremum. Then, the \ac{TV} value of $f$ is,
\begin{equation}
    J_{TV} = \sum_{i=1}^N a_if_i,
\end{equation}
where $a_1$ and $a_N$ are,
\begin{equation}\label{eq:a1L}
a_1 = 
    \begin{cases}
    1&f_1 \textrm{ is a maximum}\\
    -1&f_1 \textrm{ is a minimum}
    \end{cases}\quad
    a_N = 
    \begin{cases}
    1&f_N \textrm{ is a maximum}\\
    -1&f_N \textrm{ is a minimum}
    \end{cases},
\end{equation}
and for the rest of the indices,
\begin{equation}\label{eq:ai}
a_i = 
    \begin{cases}
    2&f_i \textrm{ is a maximum}\\
    -2&f_i \textrm{ is a minimum},
    \end{cases}
\end{equation}
where by maximum or minimum we refer to the local notions.
The value of \ac{TV} can be calculated also as the inner product between the subgradient $-p$ 
and the signal $f$. In addition, we know that the subgradient is zero when the signal is monotone and two adjacent and equal pixels do not separate (Attribute \ref{Att:adjacent}). Then, $J_{TV}$ of $f$ can be calculated by,
\begin{equation}
    J_{TV} = -\inp{p}{f}=-\sum_{i=1}^N p^im_if_i,
\end{equation}
where $p^i$ is the negative subgradient of the $i$th extremum point. By variation of parameters we get,
\begin{equation}\label{eq:subgradient}
    -p^i=\frac{a_i}{m_i}.
\end{equation}

In Algo. \ref{algo:FastSubgradient} we summarize the steps to find the subgradient. 
\begin{algorithm}[phtb!] \caption{Fast subgradient calculation}
\begin{algorithmic}[1]
        \Inputs {$f$}
		\State{Find the extrema points of $f$ and their respective number of  pixels, $\{m_i\}$.}
		\State{Find the coefficients $\{a_i\}$ in Eqs. \eqref{eq:a1L} and \eqref{eq:ai}.}
		\State{Calculate the negative subgradient, $p$, according to Eq. \eqref{eq:subgradient}.}
		\Outputs {The negative subgradient $p$}
    \end{algorithmic}
    \label{algo:FastSubgradient}
\end{algorithm}
In Fig. \ref{fig:subgradient} we illustrate Algo. \ref{algo:FastSubgradient}. We calculate the subgradients (right column) of two piecewise monotone signals (left column). The upper signal has one pixel at any extremum and the lower signal has $m_i$ pixels at the $i$th extremum.

\begin{figure}[phtb!]
\centering
\captionsetup[subfigure]{justification=centering}
\begin{subfigure}[t]{0.49\textwidth}
\centering
    \includegraphics[trim=0 0 0 0, clip,width=1\textwidth,valign = t]{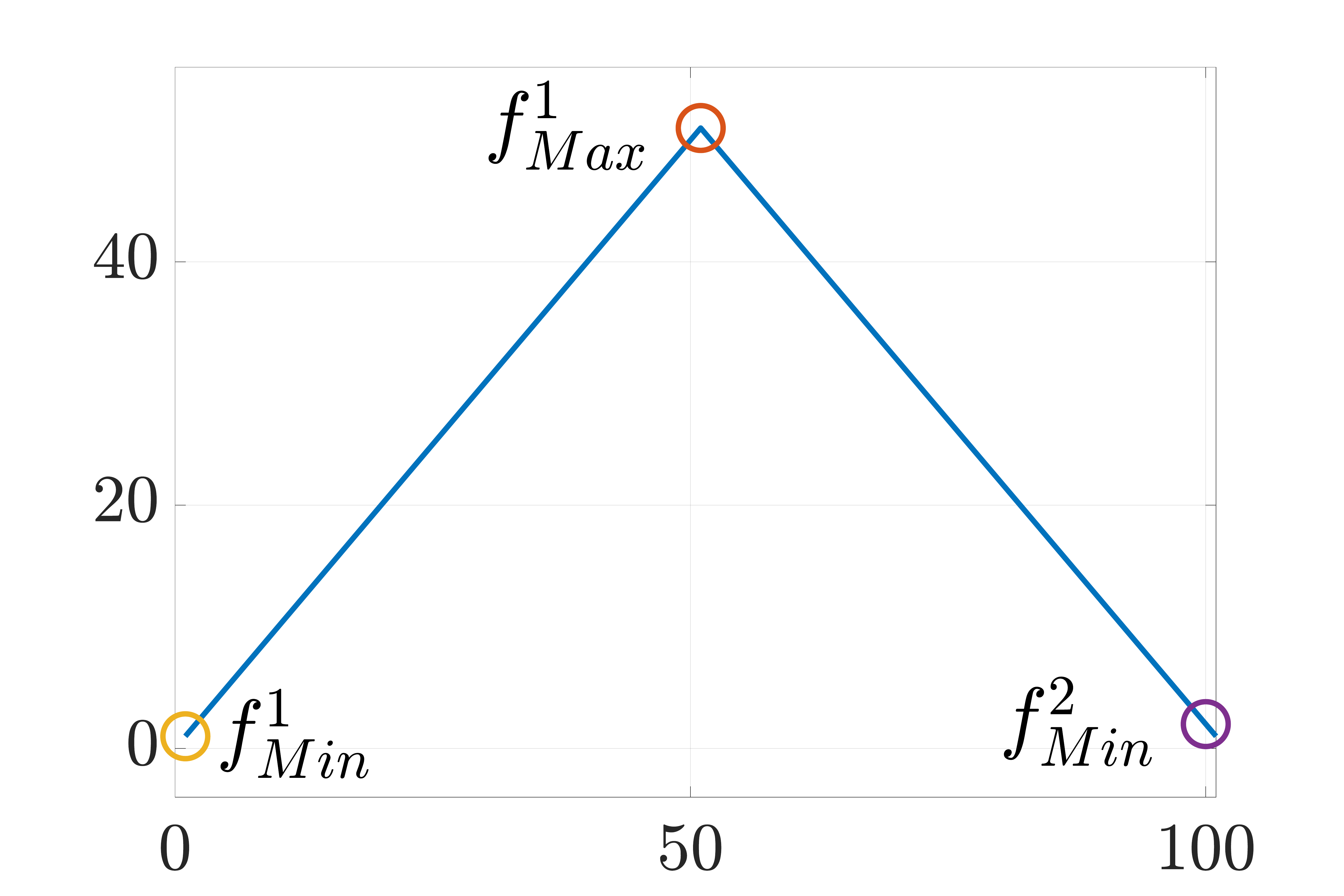}
    \caption{piecewise strictly monotone signal}
    \label{subfig:piecewiseMonotone}
\end{subfigure}
\begin{subfigure}[t]{0.49\textwidth} 
\centering
        \includegraphics[trim=0 0 0 0, clip,width=1\textwidth,valign = t]{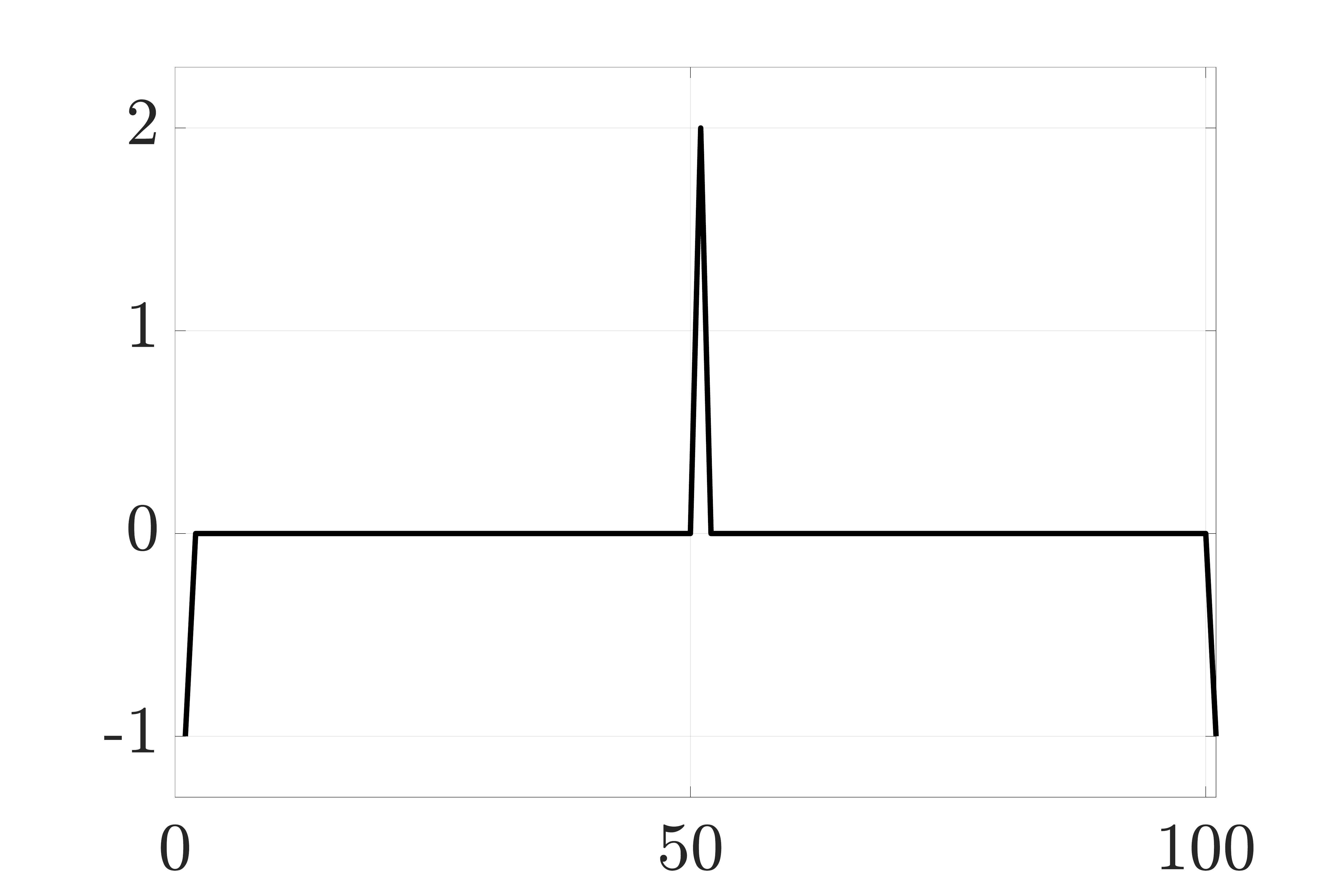}
        \caption{A subgradient of the signal in (a)}
        \label{subfig:subgradientPiecewiseMonotone}
    \end{subfigure}\\
    \begin{subfigure}[t]{0.49\textwidth}
\centering
    \includegraphics[trim=0 0 0 0, clip,width=1\textwidth,valign = t]{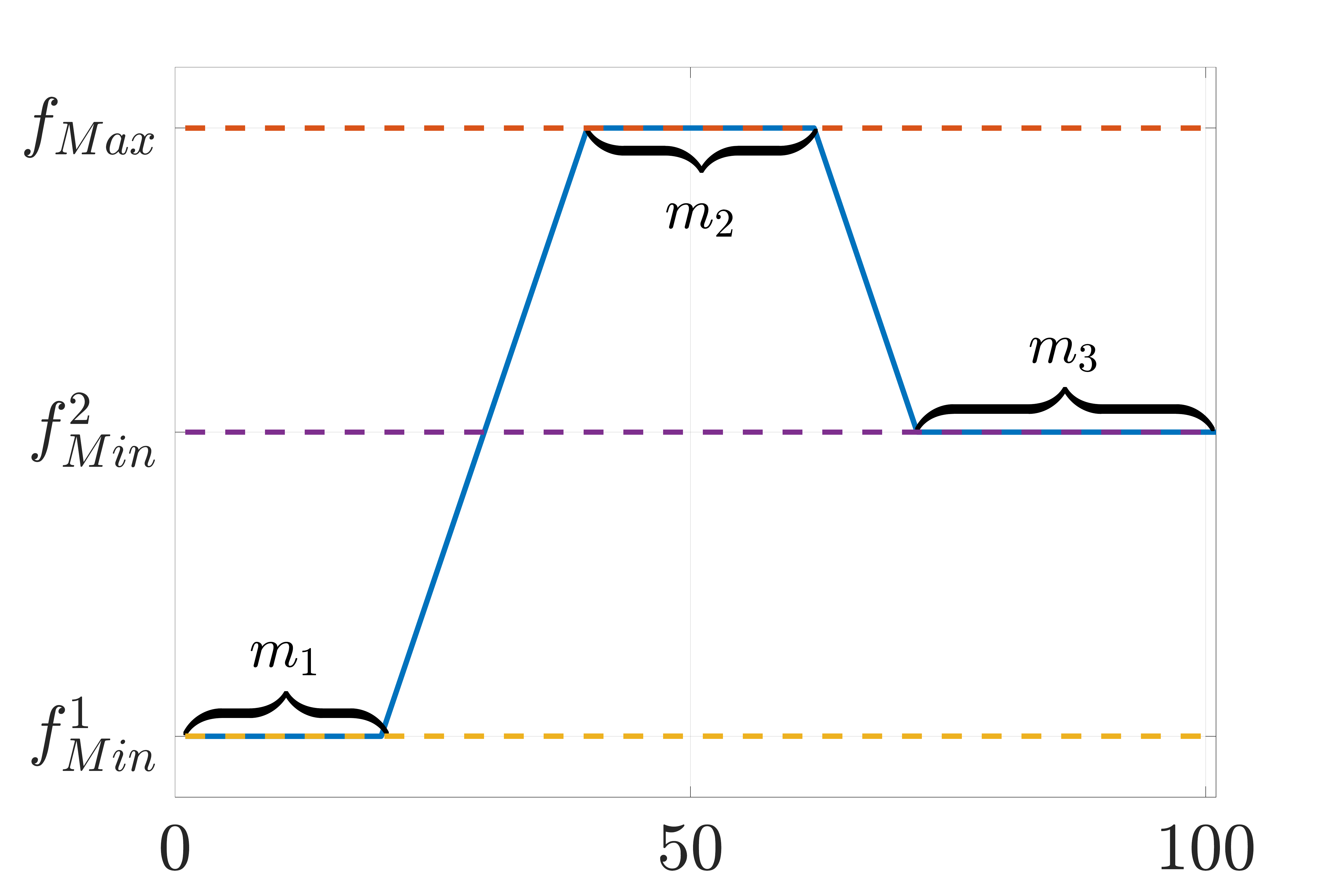}
    \caption{piecewise monotone signal}
    \label{subfig:NSpiecewiseMonotone}
\end{subfigure}
\begin{subfigure}[t]{0.49\textwidth} 
\centering
        \includegraphics[trim=0 0 0 0, clip,width=1\textwidth,valign = t]{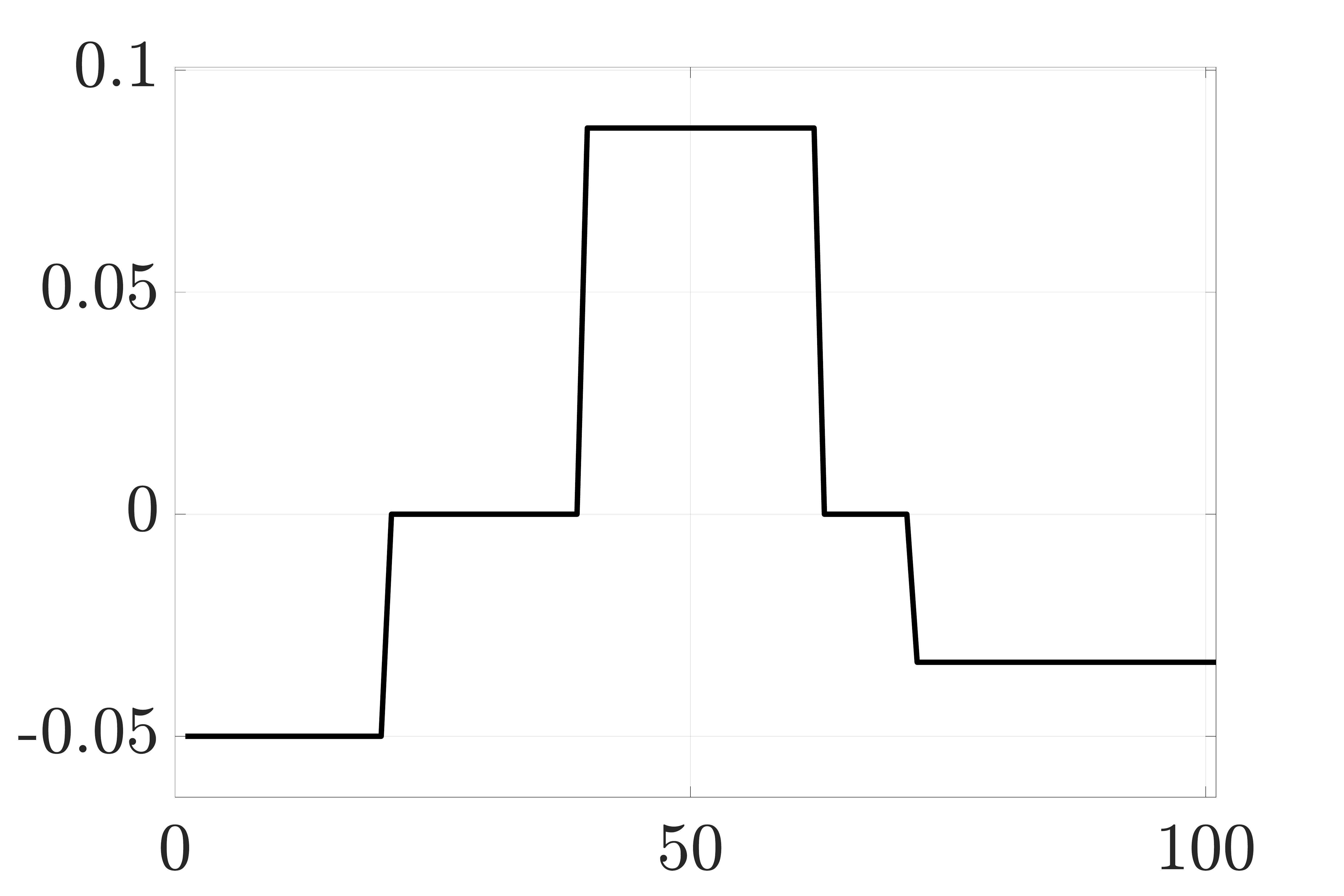}
        \caption{A subgradient of the signal in (c)}
        \label{subfig:subgradientNSPiecewiseMonotone}
    \end{subfigure}
     \caption{{\bf{Piecewise monotonic signals}} - (a) a piecewise monotone signal with one pixel at each extremum (b) its subgradient (c) a piecewise monotone signal (d) its subgradient. }
    \label{fig:subgradient}
\end{figure}

Note that the subgradient depends on  the relation between the pixels. Namely, the subgradient is valid as long as the number of pixels, $\{m_i\}_{i=1}^N$, and their respective extrema attribute (minimum or maximum) do not change.

\subsubsection{Subgradient Updating} 
Let $\mathcal{T}=\{T_i\}_{i=0}^L$ be the time points at which the subgradient changes, where $T_0=0$, and $T_L$ is the extinction time. These time points are when the relation among the extrema points changes. Equivalently, when the number of pixels in the extrema changes. An intersection between two adjacent pixels or number of pixels are termed \emph{merging event}. From \emph{Attributes} \ref{Att:merging} and \ref{Att:adjacent} we can find these time points and compute the updated subgradient. 

{\bf{Merging event prediction:}} The discrete gradient approximation is $\nabla \psi^{(j)}:=\psi^{(j+1)}-\psi^{(j)}= 0$, where $\psi^{(j)}$ is the entry $j$ of the vector $\psi$. Using Eq. \eqref{eq:GilboaSolu}, the very next merging event after $T_i$ can be calculated by
\begin{equation}\label{eq:transitiontimePoint}
    T_{i+1} = T_i+\min_{j\in\mathcal{J}^*}\{-\nabla \psi^{(j)}(T_i) / \nabla p^{(j)}(\psi(T_i))\},
\end{equation}
where 
\begin{equation*}
\mathcal{J}^*=\{j\,\,s.t.\,\, 0<-\nabla \psi^{(j)}(T_i) / \nabla p^{(j)}(\psi(T_i))<\infty\}.    
\end{equation*}

{\bf{Subgradient updating:}} According to \emph{Attribute} \ref{Att:adjacent}, the merged entries evolve together at the same pace. In addition, the subgradient at other locations is unchanged 
Since the average of the subgradient is zero (\emph{Attribute} \ref{Att:Average}), the subgradient of the merged entries is the average of the previous subgradient at these entries.
Merging event prediction and subgradient updating are detailed below and concisely formalized in Algorithm \ref{algo:FastTV}.

\begin{algorithm}[phtb!] \caption{Accelerated \ac{TV} flow}
\begin{algorithmic}[1]
		\Inputs{$f$} 
		\State {\bf{Initialize:}} $\psi_0\leftarrow f$,  $t\leftarrow 0$, $\mathcal{T}= \emptyset$, and $\mathcal{P}= \emptyset$
		\State{Calculate the negative subgradient, $p_0$, of $f$ by invoking Algo. \ref{algo:FastSubgradient}, add $p_0$ to the set $\mathcal{P}$.}
		\While{$\norm{p_i}>0$}
		\State Find the next time transition point, $T_{i+1}$ (Eq. \eqref{eq:transitiontimePoint}).
		\State $\psi_{i+1}\leftarrow \psi_i+ (T_{i+1}-T_{i}) \cdot p_i$. 
		\State \label{State:clusters}Find clusters, $\{\mathcal{M}_k\}_{k=1}^r$, where $\abs{\nabla \psi_{i+1}}=0$
		\State Update the next negative subgradient, $p_{i+1}$, such that 
		$p_{i+1}(\mathcal{M}_k)=\frac{1}{\abs{\mathcal{M}_k}}\sum_{\mathcal{M}_k}p_i(\mathcal{M}_k),\quad k=1,\cdots,r.$
		\State Add $p_{i+1}$ to the set $\mathcal{P}$; Add $T_{i+1}$ to the set $\mathcal{T}$. 
		\EndWhile
		\State {\bf{Outputs:}} $\mathcal{T},\mathcal{P}$. The set $\mathcal{P}$ contains the negative subgradient sequence and the set $\mathcal{T}$ contains the transition times. Use Eq. \eqref{eq:GilboaSolu} to get $\psi(t)$ for $t\in [0,T_L]$.
    \end{algorithmic}
    \label{algo:FastTV}
\end{algorithm}

\begin{remark}[\ac{TV} flow - physical alegory]\label{remark:TV1Dim}
For one dimensional signals, the \ac{TV} flow initiates every pixel with an initial velocity. Then, the rest of the flow is a series of pure plastic collisions of the pixels. Thus, the velocity of the center of mass is preserved and equals zero.
\end{remark}

\subsubsection{Closed form solution}
Let us assume an initial condition, $\psi(0)=f$, is orthogonal to the kernel of $J_{TV}$ (constant functions). The reconstruction of $f$ from the set of 
negative
subgradients $\{p_i\}$ is \cite{burger2016spectral}
\begin{equation}\label{eq:GilboaRecon},
    f = \psi(0) = \sum_{i=1}^L T_i(p_{i+1}-p_{i}).
\end{equation}
\begin{proposition}[Linear decay]\label{prop:closed} The solution of \eqref{eq:ss} is a sum of spectral components decaying linearly. More formally, if the initial condition, $f$, is orthogonal to the kernel set of $P$ then the solution of Eq. \eqref{eq:ss} is
\begin{equation}\label{eq:closedFormSolutionTV}
    \psi(t) = \sum_{i=1}^{L}\left(1+\lambda_i t\right)^+\varphi_i,\,\,\textrm{where}\,\,
    \lambda_i=-T_i^{-1}\,\, \textrm{and}\,\, \varphi_i=\frac{p_{i}-p_{i+1}}{\lambda_i}.
\end{equation}
\end{proposition}
\begin{proof}

Let us reformulate the solution, Eq. \eqref{eq:GilboaSolu} for the first time range $t\in[0,T_1]$. Substituting Eq. \eqref{eq:GilboaRecon} into Eq. \eqref{eq:GilboaSolu}, we have
\begin{equation*}
    \begin{split}
        \psi(t) &= \psi(0)+(t-0)p_1=\sum_{i=1}^L T_i(p_{i+1}-p_{i}) + tp_1\\
        &\quad\underbrace{=}_{p_{L+1}=0}\,\,\sum_{i=1}^L T_i(p_{i+1}-p_{i}) + t\sum_{i=1}^L(p_{i}-p_{i+1})
        =\sum_{i=1}^{L}(T_i-t)(p_{i+1}-p_{i}).
    \end{split}
\end{equation*}
Therefore, 
\begin{equation}
\psi(T_1)=\sum_{i=1}^{L}(T_i-T_1)(p_{i+1}-p_{i})=\sum_{i=2}^{L}(T_i-T_1)(p_{i+1}-p_{i}).
\end{equation}
In a similar manner, we can  reformulate the solution for $t\in[T_1,T_2)$ as, 
\begin{equation}
    \psi(t) =    \sum_{i=2}^{L}(T_i-t)(p_{i+1}-p_{i}).
\end{equation}
By induction, the general solution is,
\begin{equation}
   \psi(t) = \sum_{i=1}^{L}(T_i-t)^+(p_{i+1}-p_{i}).
\end{equation}
Denoting $\lambda_i=-T_i^{-1}$, this can be expressed as, 
\begin{equation*}
    \begin{split}
        \psi(t) &=\sum_{i=1}^{L}(-\lambda_i^{-1}-t)^+(p_{i+1}-p_{i})=\sum_{i=1}^{L}(1+\lambda_i t)^+\frac{p_{i}-p_{i+1}}{\lambda_i}.\quad \Box
    \end{split}
\end{equation*}
\end{proof}
The spectral  decomposition is computed by second order time derivative of $\psi(t)$ \cite{gilboa2014total}, 
\begin{equation}\label{eq:TV1D}
    \mathcal{G}(t) = \sum_{i=1}^L \varphi_i\cdot t\lambda_i^2\cdot\delta(1+\lambda_i\cdot t).
\end{equation}
This coincides with \emph{Attribute \ref{Att:FiniteDelta}}. A fast algorithm to find the \ac{TV} spectral decomposition is proceeding Algo. \ref{algo:FastTV}. After finding the sets $\mathcal{P}$ and $\mathcal{T}$ we can calculate the spectral components $\{\varphi_i\}$ by Eq. \eqref{eq:closedFormSolutionTV}.

The flow can also be defined (without using the operator $(\cdot)^+$) in  disjoint time intervals,
\begin{equation}\label{eq:inInterSolu}
    \psi(t) = \sum_{i=k}^{L}\left(1+\lambda_i t\right)\varphi_i, \quad \forall t\in[T_{k-1},T_k).
\end{equation}
We will use this formulation later in our analysis.

\subsection{Rescaled-\ac{DMD}}
We follow the work of \cite{cohen2020modes} where an analysis of \ac{DMD} was carried out for flows based on homogeneous operators. The homogeneity order dictates not only the decay profile but also the support in time of the solution. In particular, TV-flow decays linearly and has a finite extinction time. However, a flow linearization algorithm, such as \ac{DMD}, can be interpreted as an exponential data fitting algorithm \cite{askham2018variable} resulting in functions with infinite support. This contradiction yields an inherent error in the dynamic reconstruction by \ac{DMD}. In \cite{cohen2020modes} it was suggested to solve this problem  by time reparameterization. Introducing a new time variable $\tau$, Eq. \eqref{eq:ss} is time rescaled by the flow,
\begin{equation}\label{eq:ress2}\tag{{\bf{R-TV-flow}}}
    \psi_\tau=G(\psi)=-\frac{\inp{p}{\psi}}{\norm{p}^2}p\quad,-p\in \partial J_{TV}(\psi).
\end{equation}
Note that, $G(a\psi)=aG(\psi),\,\forall a\in\mathbb{R}$, i.e. $G$ is a one-homogeneous operator. In addition, a \ac{TV} eigenfunction is an eigenfunction of $G$, however, this eigenfunction decays exponentioally under the dynamics \eqref{eq:ress2}. Therefore, this flow rescales only the time axis whereas the spatial axis remains unchanged. Using \eqref{eq:ss} and \eqref{eq:ress2}, the relation between $t$ and $\tau$ can be derived by, 
\begin{equation*}
\begin{split}
    \frac{d}{d\tau}\psi(t(\tau))&=-\frac{\inp{p}{\psi(t(\tau))}}{\norm{p}^2}p\\
    &=-\frac{\inp{p}{\psi(t(\tau))}}{\norm{p}^2}\frac{d}{dt}\psi(t(\tau)),
\end{split}
\end{equation*}
yielding,
\begin{equation}\label{eq:ress}
    \frac{d}{d\tau}t(\tau)=-\frac{\inp{p}{\psi(t(\tau))}}{\norm{p}^2}.
\end{equation}
This ODE gets a different form in each segment, $[T_{k-1},T_k)$. 
Substituting  Eq. \eqref{eq:inInterSolu} in Eq. \eqref{eq:ress}, we have
\begin{equation*}
    \frac{d}{d\tau}t(\tau)=-\frac{\inp{\sum_{i=k}^L\lambda_i\varphi_i}{\sum_{i=k}^L\left(1+\lambda_i t(\tau)\right)\varphi_i}}{\norm{\sum_{i=k}^L\lambda_i\varphi_i}^2}=-\frac{\sum_{i=k}^L\lambda_i\norm{\varphi_i}^2}{\sum_{i=k}^L\lambda_i^2\norm{\varphi_i}^2}-t(\tau).
\end{equation*}
The solution is,
\begin{equation}\label{eq:ODESolu}
    t(\tau)=a_ke^{-\tau} -c_k,\quad c_k=\frac{\sum_{i=k}^L\lambda_i\norm{\varphi_i}^2}{\sum_{i=k}^L\lambda_i^2\norm{\varphi_i}^2},
\end{equation}
where $a_k$ depends on the initial conditions of every segment such that $t(\tau)$ is continuous (where $t(0)=0$). Then, the time points $\{T_i\}_{i=1}^L$ are mapped to $\{\tau_i\}_{i=1}^L$, accordingly.

\begin{proposition}[Main TV-flow modes]\label{prop:Main}
In every disjoint $k$th interval, $[\tau_{k-1},\tau_k)$, the solution of time reparametrizing \eqref{eq:ss}, Eq. \eqref{eq:ress2}, has two main orthogonal modes, $\xi_1^k,\xi_2^k$, with eigenvalues zero and minus one.
\end{proposition}
\begin{proof}
Substituting Eq. \eqref{eq:ODESolu} into Eq. \eqref{eq:inInterSolu}, we get
\begin{equation*}
\begin{split}
    \psi(t(\tau)) &= \sum_{i=k}^{L}\left(1+\lambda_i t(\tau)\right)\varphi_i, \qquad\qquad\qquad \forall t\in[T_{k-1},T_k)\\
    &= \sum_{i=k}^{L}\left(1+\lambda_i \left(a_ke^{-\tau} -c_k\right)\right)\varphi_i, \qquad\qquad \forall \tau\in[\tau_{k-1},\tau_k)\\
    &=\underbrace{\sum_{i=k}^{L}\varphi_i -c_k\sum_{i=k}^{L}\lambda_i \varphi_i}_{\xi_1^k} + e^{-\tau}\underbrace{a_k\sum_{i=k}^{L}\lambda_i \varphi_i}_{\xi_2^k}=\xi_1^k+e^{-\tau}\xi_2^k, \, \forall \tau\in[\tau_{k-1},\tau_k).
    \end{split}
\end{equation*}
By plugging $c_k$ from Eq. \eqref{eq:ODESolu} into $\xi_1^k,\xi_2^k$ their orthogonality is concluded immediately.$\quad\Box$ 
\end{proof}

\subsection{Analysis of the Rescaled-\ac{DMD}}
Here, we show a closed form solution to the time Rescaled-\ac{DMD} (R-DMD).  The common thread in the following discussion is \emph{Attribute \ref{Att:orthogonal}}, the orthogonality of the \ac{TV}-spectral components. The method is summarized in Algorithm \ref{algo:DMDTV}.
\begin{theorem}[R-\ac{DMD} of TV-flow]
Let $\tau_0$ be zero, then for the interval, $[\tau_{k-1},\tau_k)$, where $k=1,\dots,L-1$, R-\ac{DMD} reveals two non-zero orthogonal modes that reconstruct accurately the TV-flow in this interval. For the last interval, $[\tau_{L-1},\tau_L)$, there is only one nonzero mode.
\end{theorem}
\begin{proof} According to Prop. \ref{prop:Main} and since \ac{DMD} is an exponential data fitting algorithm, the \ac{DMD} of the dynamics, Eq. \eqref{eq:ress2}, is as follows. The modes are
$\phi_1^k={\xi_1^k}/{\norm{\xi_1^k}},\,\phi_2^k={\xi_2^k}/{\norm{\xi_2^k}}$, and the coefficients are $\alpha_1^k = \norm{\xi_1^k}$ and $\alpha_2^k = \norm{\xi_2^k}$. Note that one mode is constant with respect to time and the second decays exponentially. Therefore, the eigenvalues are $\mu_1^k=1$ for the constant mode and $\mu_2^k=e^{-dt}$ where $dt$ is the sampling step size (see Algo. \ref{algo:DMDTV}). $\quad\Box$
\end{proof}
\begin{algorithm}[htb!] \caption{R-\ac{DMD} for \ac{TV}-flow}
\begin{algorithmic}[1]
		\State {\bf{Inputs:}} The initial condition $f$, and sampling step size  $dt$.
		\State {\bf{Initialize:}} Evolve the solution of \eqref{eq:ress2} uniformly with a step size $dt$.
		\State Invoke Algo. \ref{algo:FastTV} with the input $f$ - the result is $\mathcal{T}$ and $\mathcal{P}$.
		\State Map the set of transition time points, $\mathcal{T}$, to a new set $\Hat{\mathcal{T}}$ (Eq. \eqref{eq:ress}).
		\For {Every time segment $[\tau_i,\tau_{i+1}),\quad \tau_i,\tau_{i+1}\in \Hat{\mathcal{T}}$}
		\State Invoke the classic \ac{DMD} with $r=2$ (when $i=L-1$, $r=1$) \cite{schmid2010dynamic}.
		\EndFor
		\State {\bf{Outputs:}} Modes $\{\phi_1^k,\phi_2^k\}_{k=1}^L$, coefficients $\{\alpha_1^k,\alpha_2^k\}_{k=1}^L$, and eigenvalues $\mu_1^k=1,\mu_2^k=e^{-dt}$.
    \end{algorithmic}
    \label{algo:DMDTV}
\end{algorithm}

Now we formulate the relation between the \ac{TV} spectral components $\varphi_k$ and the R-DMD modes.
\begin{proposition}[Revealing \ac{TV} spectral components from R-\ac{DMD}]\label{prop:componentsAndModes}
Given the result of Algo. \ref{algo:DMDTV}, we can formulate the $k$th spectral component, $\varphi_k$, (Eq. \eqref{eq:closedFormSolutionTV}) by the following relation,
    \begin{equation}
        a_k\lambda_k\varphi_k =\alpha_2^k\phi_2^k - \frac{\inp{\alpha_2^k\phi_2^k}{\alpha_2^{k+1}\phi_2^{k+1}}}{\norm{\alpha_2^{k+1}\phi_2^{k+1}}^2}\alpha_2^{k+1}\phi_2^{k+1}.
    \end{equation}
\end{proposition}
\begin{proof}  
\begin{equation*}
\begin{split}
    \alpha_2^k\phi_2^k -& \frac{\inp{\alpha_2^k\phi_2^k}{\alpha_2^{k+1}\phi_2^{k+1}}}{\norm{\alpha_2^{k+1}\phi_2^{k+1}}^2}\alpha_2^{k+1}\phi_2^{k+1}= \xi_2^k - \frac{\inp{\xi_2^k}{\xi_2^{k+1}}}{\norm{\xi_2^{k+1}}^2}\xi_2^{k+1}=\\
    &=a_{k}\sum_{i=k}^{L}\lambda_i \varphi_i - \frac{\inp{a_k\sum_{i=k}^{L}\lambda_i \varphi_i}{a_{k+1}\sum_{i=k+1}^{L}\lambda_i \varphi_i}}{\norm{a_{k+1}\sum_{i=k+1}^{L}\lambda_i \varphi_i}^2}a_{k+1}\sum_{i=k+1}^{L}\lambda_i \varphi_i\\
    & =a_{k}\lambda_k \varphi_k+\\
    &\qquad a_{k}\sum_{i=k+1}^{L}\lambda_i \varphi_i - a_k\frac{\inp{\lambda_k \varphi_k+\sum_{i=k+1}^{L}\lambda_i \varphi_i}{\sum_{i=k+1}^{L}\lambda_i \varphi_i}}{\norm{\sum_{i=k+1}^{L}\lambda_i \varphi_i}^2}\sum_{i=k+1}^{L}\lambda_i \varphi_i\\
    &=a_k\lambda_k\varphi_k. \quad \Box
\end{split}
\end{equation*}

\end{proof}


\subsection{Decay profile decomposition with Koopman modes of the \acs{TV} flow }\label{subsec:DecayProfileTVKMD}
\ac{RDMD} can be a solution for homogeneous flows when the dynamics belongs to $C^1$. However, it is limited when the dynamics is in $C^0$ almost everywhere \cite{cohen2021examining}. The attribute of $C^0$ a.e. is equivalent to the fact that the Koopman modes do not exist during the entire dynamics.

Under the assumption that the dynamical system has a typical monotonic decay profile, a new method was suggested to find the Koopman modes \cite{cohen2021examining}. We summarize this algorithm in Section \ref{subsec:DecayProfileKMD}. Now, we would like to apply this method on \ac{TV}-flow.
The typical decay profile of zero-homogeneous flows, such as \ac{TV}-flow, is a linear function (see Eq. \eqref{eq:closedFormSolutionTV})
,  and can be formulated as
\begin{equation}
a_{\lambda}(t)=(1+\lambda t)^+.
\end{equation}
Denoting the step size as $dt$, we can formulate the dictionary, $D$, as
\begin{equation}
D=\begin{bmatrix}
1&(1+\lambda_1 dt)&\cdots\\
&\vdots&\\
1&(1+\lambda_N dt)&\cdots
\end{bmatrix}.
\end{equation}
The goal is to find a sparse matrix $V$ such that 
\begin{equation}\label{eq:DecayModeApprox}
\Psi\approx VD.
\end{equation}

\section{A Two Dimensional Approximation of TV-flow}

We now suggest a way to extend our fast algorithm (FGD) to two dimensions, approximating the anisotropic flow. Anisotropic \ac{TV} plays an important role in mitigating edge sparsity required by many image processing applications, such as deconvolution, denoising, MRI construction \cite{lou2015weighted}, and image segmentation \cite{bui2021weighted,wu2021adaptive}.
There are previous approximations of 2D anisotropic \ac{TV} minimizations, such as \cite{choksi2011anisotropic}.
However, there has been little research on fast TV-flow. Until now, a common practice is to solve a nonsmooth convex minimization problem at each time step (using \acf{ADMM}, primal dual or other methods). This is very inefficient, naturally.
Applying our proposed \ac{FGD} algorithm can accelerate the process tremendously. 

{\bf{Anisotropic \texorpdfstring{\ac{TV}}{TEXT}}.}
The anisotropic \ac{TV} is defined as follows,
\begin{equation}
    J_{ani}(\psi) = \inp{\norm{\nabla \psi}_1}{1}=\inp{\abs{D_x\psi}+\abs{D_y\psi}}{1}.
\end{equation}

We can calculate the subgradient of every column and row with Algo. \ref{algo:FastSubgradient}. However, some of the properties of the 1D \ac{TV}-flow are not valid in 2D. For example, the computation for the next merge event in a row is corrupted by the subgradient of a column. Therefore, we cannot update the subgradient with Algo. \ref{algo:FastTV} and we should repeatedly compute the subgraident with Algo. \ref{algo:FastSubgradient}.

Let $P^{(x)}$ be a matrix whose rows are the corresponding negative subgradients of the rows of $\psi$. Similarly, $P^{(y)}$ contains the negative subgradients of the columns of $\psi$. We assume that $P^{(x)}+P^{(y)}=0$ for every entry iff $P^{(x)}=0$ and $P^{(y)}=0$. Namely, it is most likely that the gradient descent flows of \ac{TV} in $x$ and $y$ axes do not exactly cancel each other. The explicit scheme of the gradient descent flow is given by,
\begin{equation}\label{eq:explicitScheme2D}
\psi_{k+1} =\psi_k+\left(P_k^{(x)}+P_k^{(y)}\right)dt_k, \quad \psi_0=f,
\end{equation}
where $P_k^{(x)}$ and $P_k^{(y)}$ denote the subgradients of $\psi_k$. 
The constraint on the step size $dt_k$ for which this explicit scheme converges is formulated in the following theorem.

\begin{theorem}[Convergence of the explicit scheme]\label{theo:convergence}
Let us define the following ratio,
\begin{equation}
\tilde{\lambda}_k=\frac{\norm{P_k^{(x)}+P_k^{(y)}}^2}{J_{ani}(\psi_k)},
\end{equation}
where $J_{ani}(\psi_k)$ is the anisotropic \ac{TV} of $\psi_k$. For any
time step $dt_k$ admitting,
\begin{equation}\label{eq:stepSizeIso}
dt_k=\frac{\delta}{\tilde{\lambda}_k},
\end{equation} 
where $\delta \in (0,2)$, the explicit scheme, Eq. \eqref{eq:explicitScheme2D}, converges to steady state.
\end{theorem}
\begin{proof}
Let us examine the evolution of the $\ell^2$ norm of $\psi_k$ under the anisotopic \ac{TV}-flow. We can formulate the norm of $\psi_{k+1}$ as
\begin{equation} 
\begin{split}
\norm{\psi_{k+1}}^2&=\norm{\psi_{k}+\left(P_k^{(x)}+P_k^{(y)}\right)dt_k}^2\\
&=\norm{\psi_{k}}^2+2\inp{\psi_{k}}{P_k^{(x)}+P_k^{(y)}}dt_k+\norm{\left(P_k^{(x)}+P_k^{(y)}\right)}^2\cdot dt_k^2\\
&=\norm{\psi_{k}}^2+2\inp{\psi_{k}}{P_k^{(x)}+P_k^{(y)}}\delta\frac{J_{ani}(\psi_k)}{\norm{P_k^{(x)}+P_k^{(y)}}^2}\\
&\qquad+\norm{P_k^{(x)}+P_k^{(y)}}^2\cdot \delta^2\frac{J(\psi_k)^2}{\norm{P_k^{(x)}+P_k^{(y)}}^4}\\
&=\norm{\psi_{k}}^2-2\delta\frac{J_{ani}(\psi_k)^2}{\norm{P_k^{(x)}+P_k^{(y)}}^2}+ \delta^2\frac{J_{ani}(\psi_k)^2}{\norm{P_k^{(x)}+P_k^{(y)}}^2}\\
&=\norm{\psi_{k}}^2+\left(\delta^2-2\delta\right)\frac{J_{ani}(\psi_k)^2}{\norm{P_k^{(x)}+P_k^{(y)}}^2}
\end{split}
\end{equation}
Since $\delta\in(0,2)$, $\norm{\psi_{k+1}}^2-\norm{\psi_{k}}^2\leq0$. Then, the series $\{\norm{\psi_k}\}^2$ is monotonically decreasing and bounded from below by zero, therefore converges. Then, the difference between two successive elements converges to zero. Since the term $\delta^2-2\delta$ is constant then the ratio $\frac{J_{ani}(\psi_k)^2}{\norm{P_k^{(x)}+P_k^{(y)}}^2}$ converges to zero. In addition, the denominator is bounded from above, therefore, $J(\psi_k)\to 0$ as $k\to\infty$. 
\end{proof}

\section{Results}\label{sec:res}
In this section, we illustrate the theory and algorithms discussed above. We use standard first-order discretization of the derivatives and Neumann boundary conditions.
\subsection{1D results}
We begin with a toy example, depicted in Fig. \ref{subfig:peaks_src}. We show that the solution of \eqref{eq:ss} decays linearly (Fig. \ref{subfig:pulse_flow_mesh}) and that of Eq. \eqref{eq:ress2} piecewise exponentially, (Fig. \ref{subfig:pulse_flow_dmd_mesh}).
\begin{figure}[phtb!]
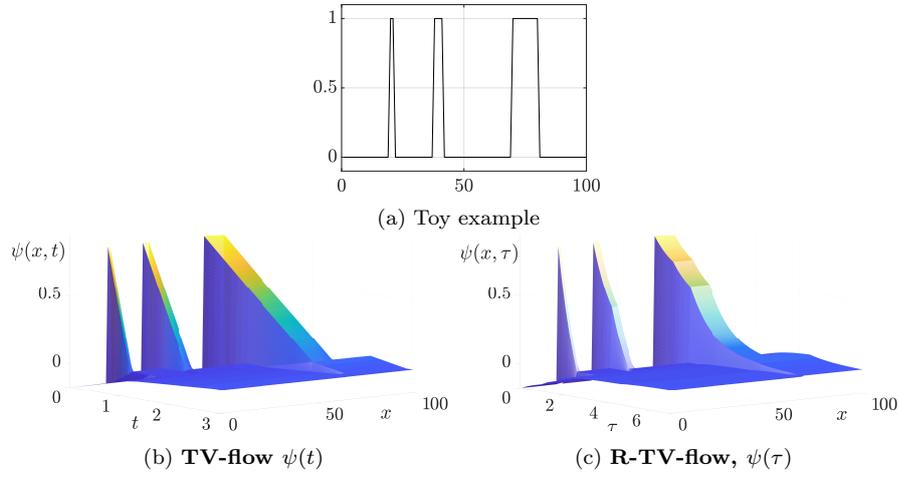

\centering
\begin{subfigure}[t]{0.35\textwidth} 
\centering
        \includegraphics[trim=0 125 0 300, clip,width=1\textwidth,valign = t]{./Figs/1D/peaks_src.pdf}
        \caption{Toy example}
        \label{subfig:peaks_src}
    \end{subfigure}\\
\begin{subfigure}[t]{0.49\textwidth}
\centering
\includegraphics[trim=0 190 0 800, clip,width=1\textwidth]{./Figs/1D/pulse_flow_mesh.pdf}
\caption{{\bf{\ac{TV}-flow $\psi(t)$}}}
\label{subfig:pulse_flow_mesh}
\end{subfigure}
\begin{subfigure}[t]{0.49\textwidth}
\includegraphics[trim=0 190 0 800, clip,width=1\textwidth]{./Figs/1D/DMD/pulse_flow_dmd_mesh.pdf}
\caption{{\bf{R-\ac{TV}-flow, $\psi(\tau)$}}}
\label{subfig:pulse_flow_dmd_mesh}
\end{subfigure}
\caption{{\bf{Time reparametrization}} - (a) {\bf{The initial condition}} is a signal with three pulses with different widths. (b) {\bf{\ac{TV}-flow $\psi(t)$}} decays piecewise linearly. (c) {\bf{R-\ac{TV}-flow}}, $\psi(t)$ is mapped to a piecewise smooth function. The non-smooth points represent transitions in the subgradient.}
    \label{fig:meshmesh}
\end{figure}

In Fig. \ref{fig:DMD_mods_1_and_2} we show the \ac{TV}-modes defined in Prop. \ref{prop:Main} with the initial condition Fig. \ref{subfig:peaks_src}. It contains six disjoint intervals with the corresponding modes $\{\xi_1^k,\xi_2^k\}_{k=1}^6$.
\begin{figure}[phtb!]
    \centering 
    \includegraphics[trim=850 1000 650 120, clip,width=1\textwidth]{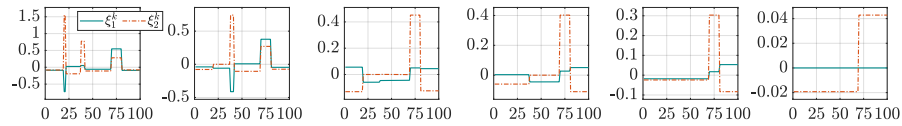}
    \caption{Modes: $\xi_1^k$ (teal) - constant, $\xi_2^k$ (orange) - exponentially decaying.
    }
    \label{fig:DMD_mods_1_and_2}
\end{figure}

In Fig. \ref{fig:fast_TV_pulse}-top the \ac{TV}-spectral decomposition, dashed red line (computed in the standard way, see \cite{gilboa2014total}) is compared with two algorithms: Algorithm \ref{algo:DMDTV} based on Prop. \ref{prop:componentsAndModes}, black dotted line, and Algorithm \ref{algo:FastTV}, blue line.
The errors between the \ac{TV}-spectral decomposition and Algorithms \ref{algo:DMDTV} and \ref{algo:FastTV} are depicted in Fig. \ref{fig:fast_TV_pulse}-bottom. We can observe an excellent match in both cases.
\begin{figure}[phtb!]
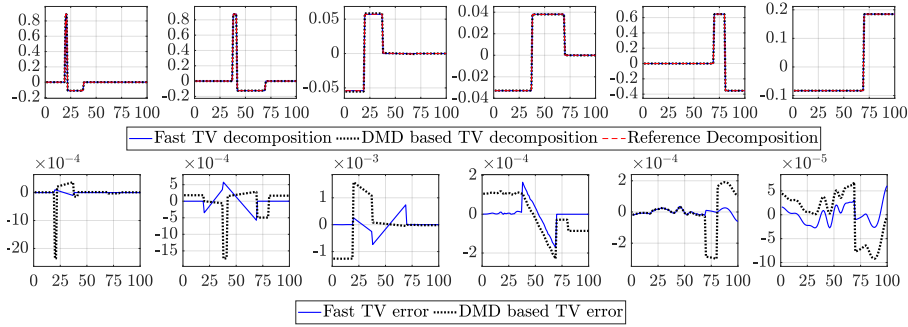

    \centering 
    \includegraphics[trim=850 800 650 120, clip,width=1\textwidth]{./Figs/1D/FastTV/fast_TV_pulse_flow.pdf}
    \centering 
    \includegraphics[trim=850 500 650 0, clip,width=1\textwidth]{./Figs/1D/FastTV/fast_TV_pulse_error.pdf}
    \caption{{\bf{TV spectral decomposition comparison for toy example}} - The standard method of spectral decomposition (Dashed red line) vs. fast TV decomposition and R-DMD decomposition in blue and dotted black lines respectively. Bottom row - respecive errors}
    \label{fig:fast_TV_pulse}
\end{figure}

In Fig. \ref{fig:fast_TV_zebrarow}, we show results of the fast \ac{TV}-spectral decomposition, Algorithm \ref{algo:FastTV}, applied on a natural signal. We arbitrarily chose the red line from the zebra in Fig. \ref{subfig:zebra_imgRedLine}, depicted in Fig. \ref{subfig:zebra_line_src}. Bands of standard \ac{TV}-spectral decomposition and the fast \ac{TV}-spectral decomposition are shown in Fig. \ref{subfig:fast_TV_zebrarow}. One can observe that our proposed fast method recovers the spectral bands faithfully, with negligible error, Fig. \ref{subfig:fast_TV_zebrarow_error}.

\begin{figure}[phtb!]
\centering
\captionsetup[subfigure]{justification=centering}

\begin{subfigure}[t]{0.2\textwidth} 
\centering
        \includegraphics[trim=0 -60 0 0, clip,width=0.473\textwidth,valign = t]{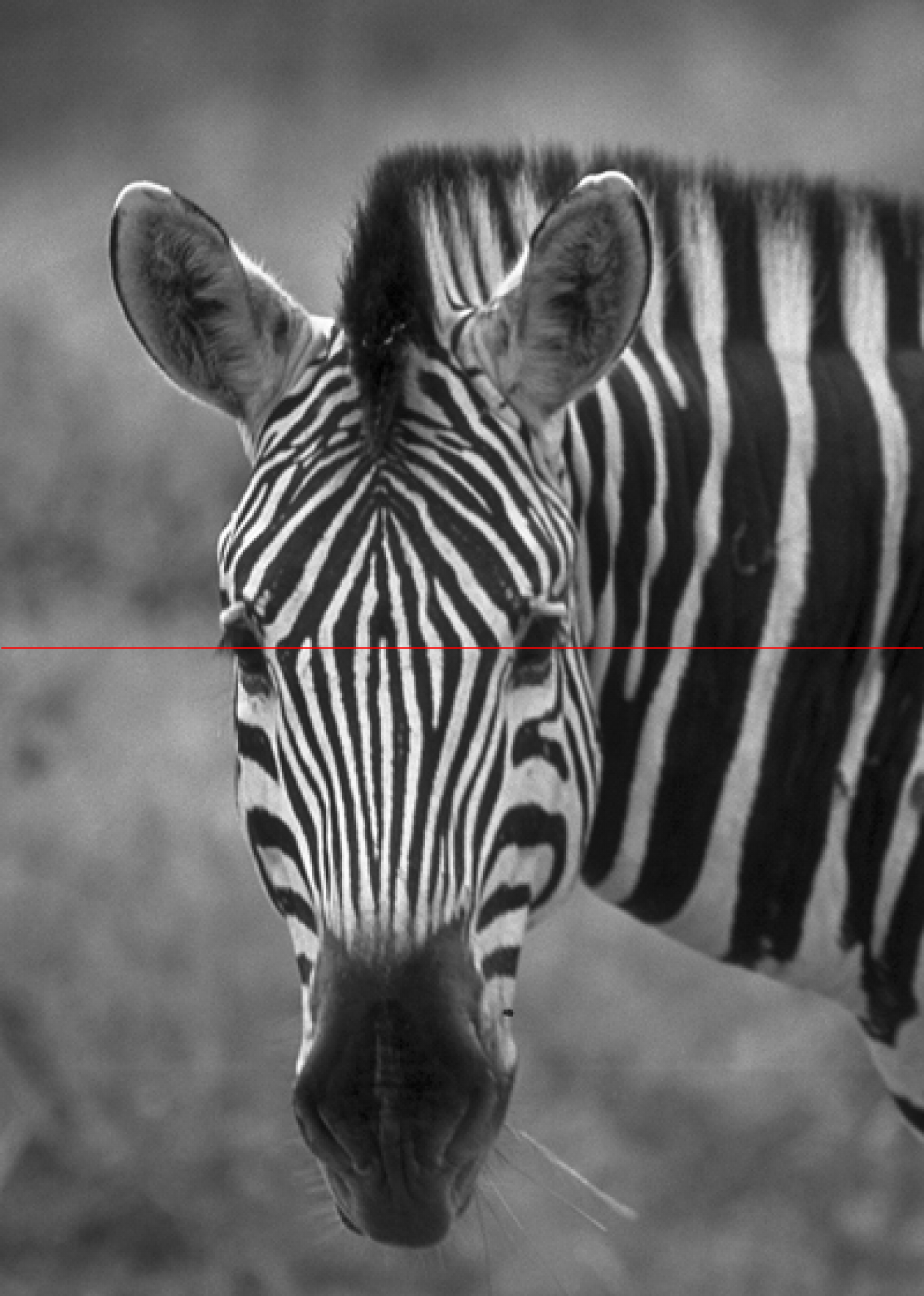}
        \caption{{\bf{Zebra image}}}
        \label{subfig:zebra_imgRedLine}
    \end{subfigure}
\begin{subfigure}[t]{0.375\textwidth}
\centering
    \includegraphics[trim=0 125 0 300, clip,width=0.66\textwidth,valign = t]{./Figs/1D/zebra_line_src.pdf}
    \caption{{\bf{The red line from (a)}}}
    \label{subfig:zebra_line_src}
\end{subfigure}\\
\begin{subfigure}[t]{1\textwidth}
\centering
    \includegraphics[trim=850 725 650 100, clip,width=1\textwidth]{./Figs/1D/FastTV/fast_TV_zebrarow.pdf}
    \caption{\bf{\ac{TV}-spectral decomposition}}
    \label{subfig:fast_TV_zebrarow}
    \end{subfigure}\\
\begin{subfigure}[t]{1\textwidth}
\centering    
    \includegraphics[trim=850 850 650 0, clip,width=1\textwidth]{./Figs/1D/FastTV/fast_TV_zebrarow_error.pdf}
    \caption{{\bf{Error}}}
    \label{subfig:fast_TV_zebrarow_error}
    \end{subfigure}
    \caption{{\bf{TV spectral decomposition comparison for an arbitrary initial condition}} - (b) The corresponding values of the pixels on the red line in (a). (c) Standard method of spectral decomposition (Dashed red line) vs. fast TV decomposition (Blue line). (d) The respective error.  }
    \label{fig:fast_TV_zebrarow}
\end{figure}


{\bf{Rough performance comparison.}} 
We report the elapsed time in seconds, running in Matlab 2018b on an 8th Gen. Core i7 laptop with 16GB RAM.  
Initial condition Fig. \ref{subfig:peaks_src}: Standard method (iterative application of \cite{chambolle2004algorithm}) - $488.4$s; ours - $0.013$s.
Initial condition Fig. \ref{subfig:zebra_line_src} (zebra):
Standard method - $7.1\times10^3$s; ours - $0.15$s. 

\subsection{2D results} 
Here we compare between isotropic and anisotropic \ac{TV}-spectral decomposition and between the \ac{TV}-spectral decomposition and the decay profile decomposition with Koopman modes \eqref{eq:DecayModeApprox}. The Anisotropic \ac{TV}-Flow is computed by the accelerated flow based on Theorem \ref{theo:convergence}. The algorithm to approximate the Koopman mode is dictated by the typical decay profile and formulated in \ref{subsec:DecayProfileTVKMD}.

In Fig. \ref{subfig:TV-zebra}, we show the isotropic \ac{TV} spectral decomposition of the Zebra from Fig. \ref{subfig:zebra_imgRedLine}. In Figs. \ref{subfig:z1_TV}--\ref{subfig:z4_TV}, we show the decomposition according to the percentage depicted in Fig. \ref{subfig:TV-zebra}.
\begin{figure}[phtb]
\centering
\subfloat[The isotropic \ac{TV}-spectrum vs. time]
  {
\includegraphics[width=0.6425\textwidth]{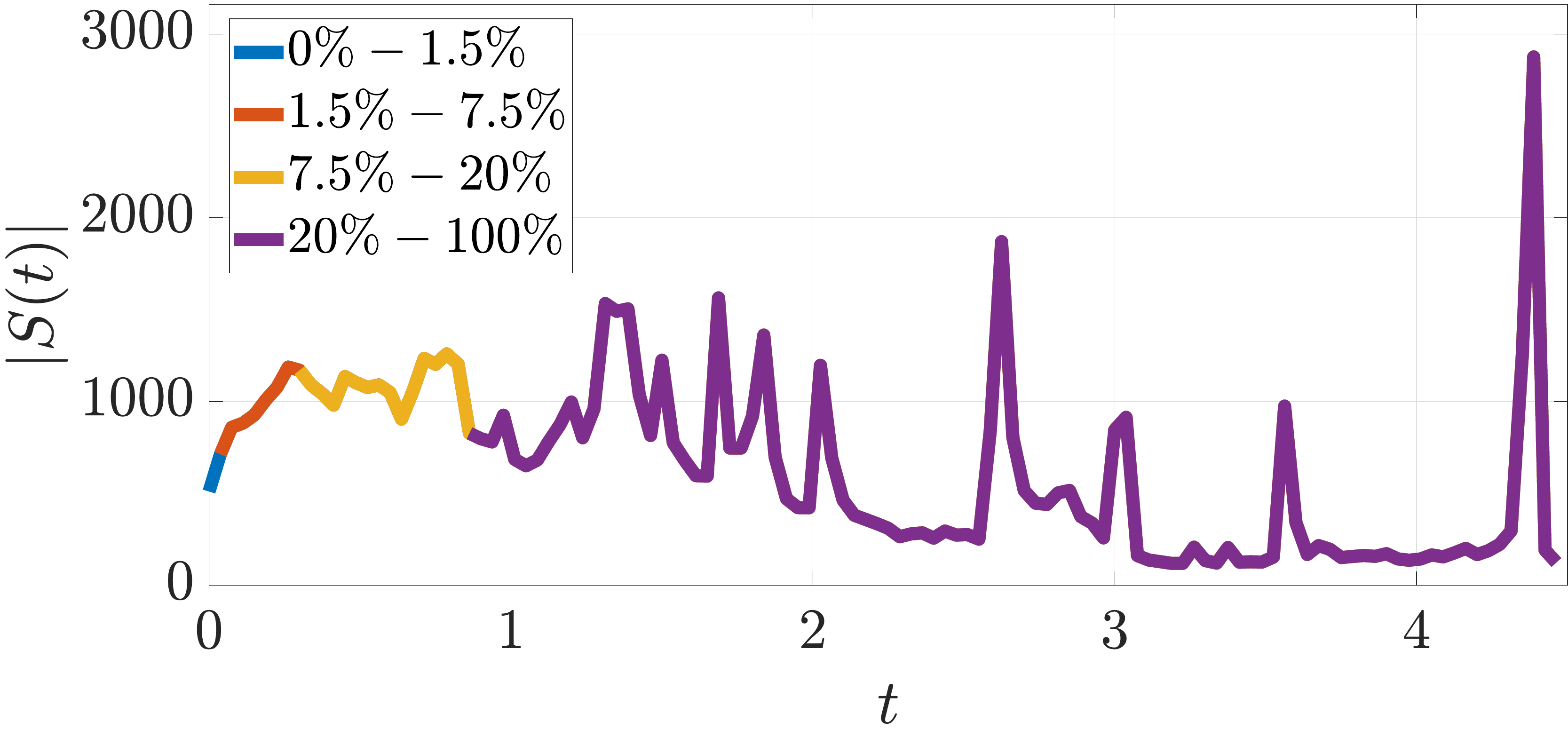}
\label{subfig:TV-zebra}
}\\
\subfloat[$0\%-1.5\%$]
  {
\includegraphics[width=0.215\textwidth]{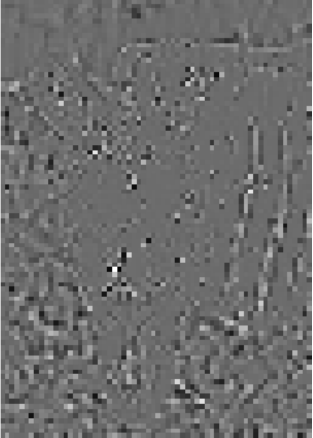}
\label{subfig:z1_TV}
}
\subfloat[$1.5\%-7.5\%$]
  {
\includegraphics[width=0.215\textwidth]{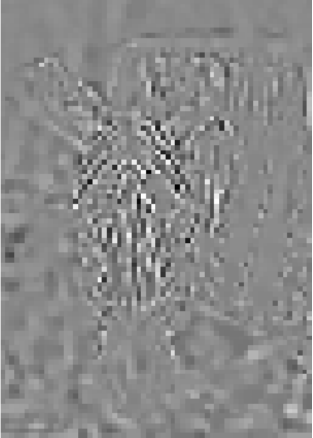}
\label{subfig:z2_TV}
}
\subfloat[$7.5\%-20\%$]
  {
\includegraphics[width=0.215\textwidth]{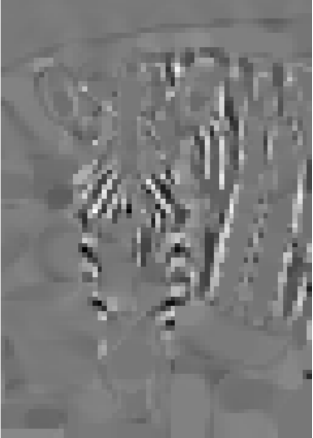}
\label{subfig:z3_TV}
}
\subfloat[$20\%-100\%$]
  {
\includegraphics[width=0.215\textwidth]{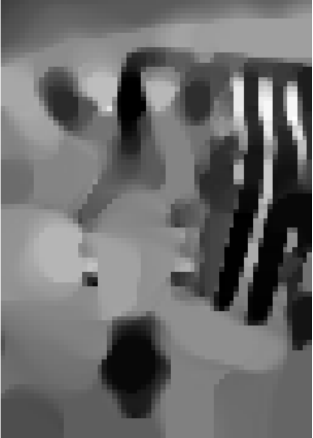}
\label{subfig:z4_TV}
}
\caption{Zebra image isotropic \ac{TV}-decomposition}
\label{Fig:ZebraDecomposition_TV}
\end{figure}
In Fig. \ref{Fig:ZebraDecomposition_TVaniso_KMD}, we present the decomposition results to the same band separation, but this time based on Eq. \eqref{eq:DecayModeApprox}. 
From a qualitative perspective, a very similar decomposition is obtained.

\begin{figure}[phtb]
\centering
\subfloat[$0\%-1.5\%$]
  {
\includegraphics[width=0.215\textwidth]{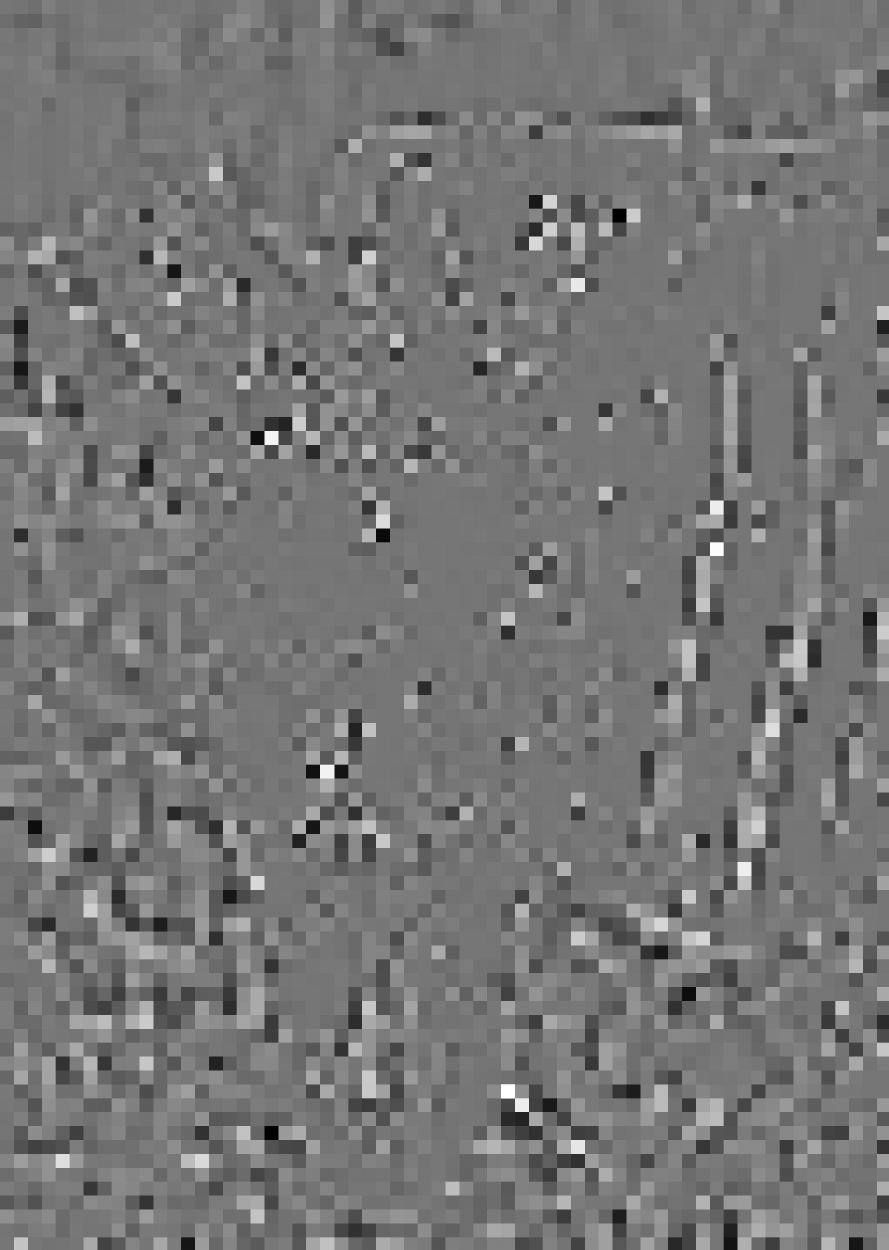}
\label{subfig:TV2D_KMD_F1}
}
\subfloat[$1.5\%-5\%$]
  {
\includegraphics[width=0.215\textwidth]{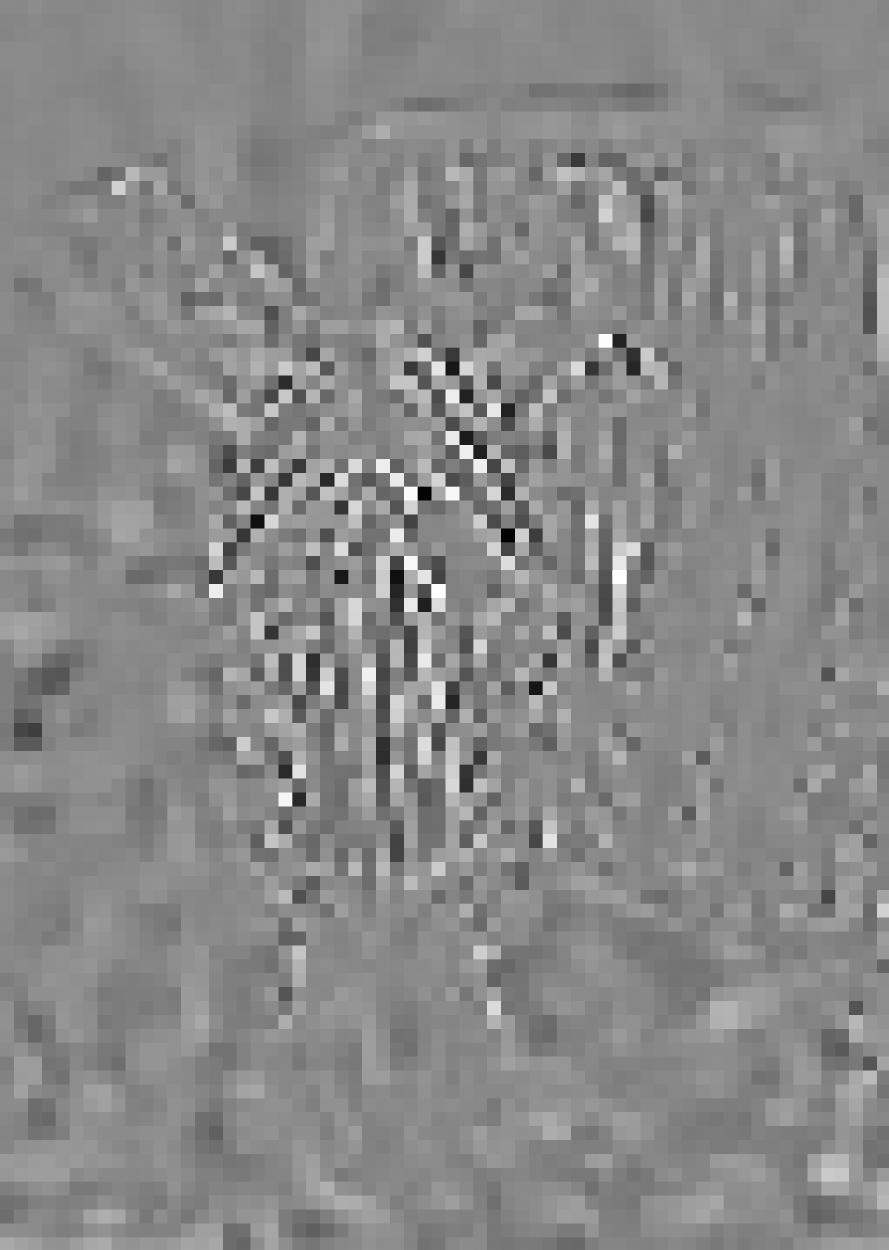}
\label{subfig:TV2D_KMD_F2}
}
\subfloat[$5\%-12.5\%$]
  {
\includegraphics[width=0.215\textwidth]{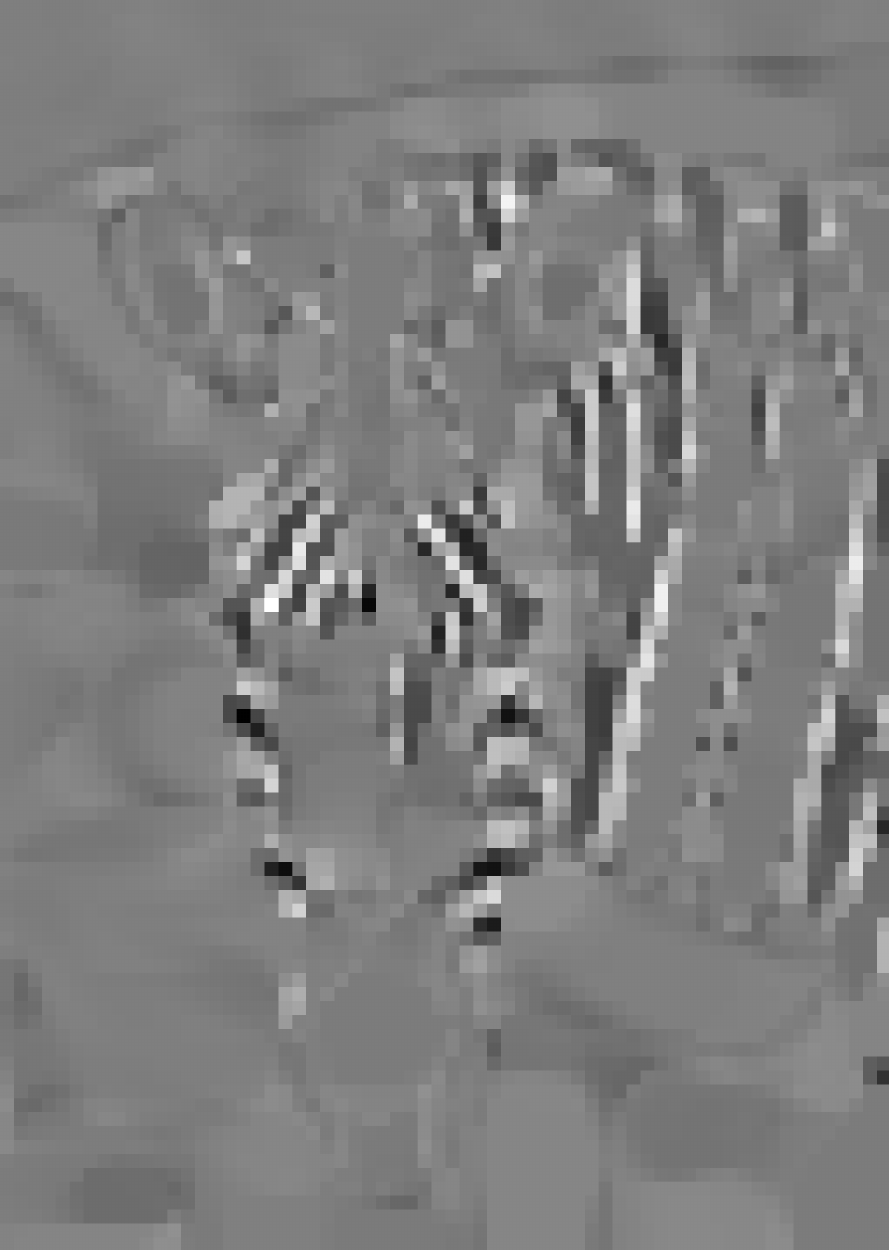}
\label{subfig:TV2D_KMD_F3}
}
\subfloat[$20\%-100\%$]
  {
\includegraphics[width=0.215\textwidth]{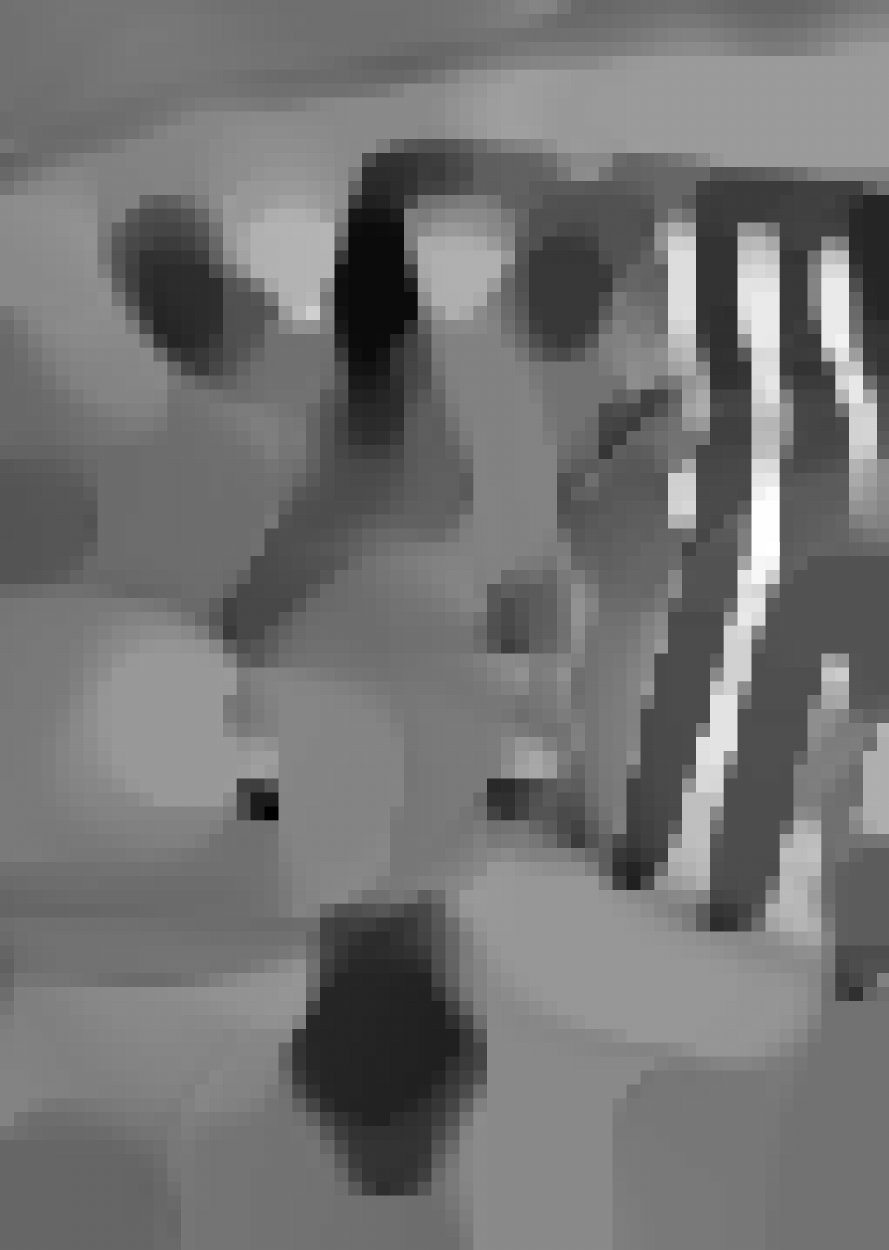}
\label{subfig:TV2D_KMD_F4}
}
\caption{Zebra image - decay profile decomposition with Koopman modes of isotropic \ac{TV} flow}
\label{Fig:ZebraDecomposition_TVaniso_KMD}
\end{figure}

In Fig. \ref{subfig:s}, the anisotropic \ac{TV}-spectrum is presented and the decomposition according to the percentage in \ref{subfig:s} is depicted in Figs. \ref{subfig:FastTV2D_F1}--\ref{subfig:FastTV2D_F4}.
\begin{figure}[phtb]
\centering
\subfloat[The anisotropic \ac{TV}-spectrum vs. time]
  {
\includegraphics[width=0.6425\textwidth]{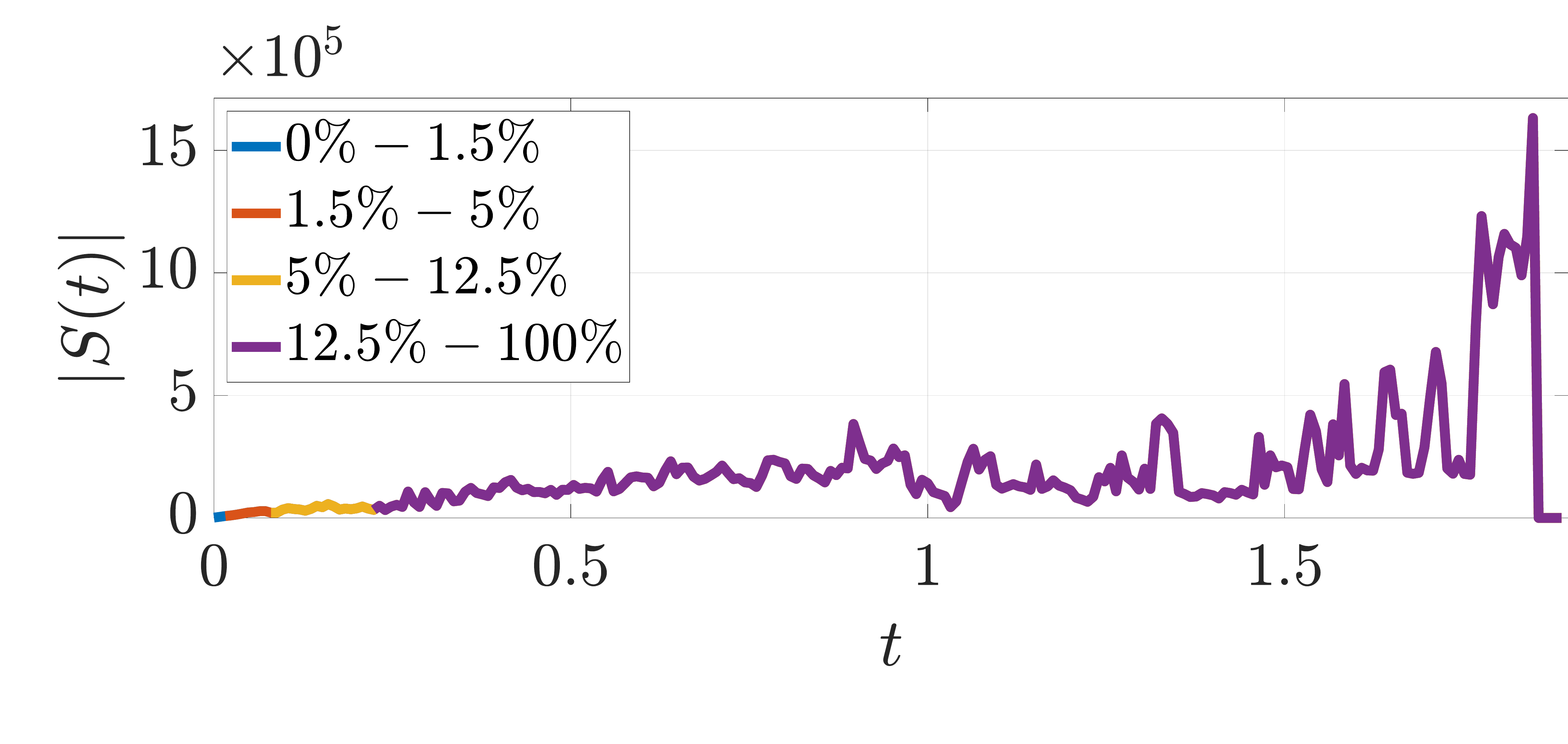}
\label{subfig:s}
}\\
\subfloat[$0\%-1.5\%$]
  {
\includegraphics[width=0.215\textwidth]{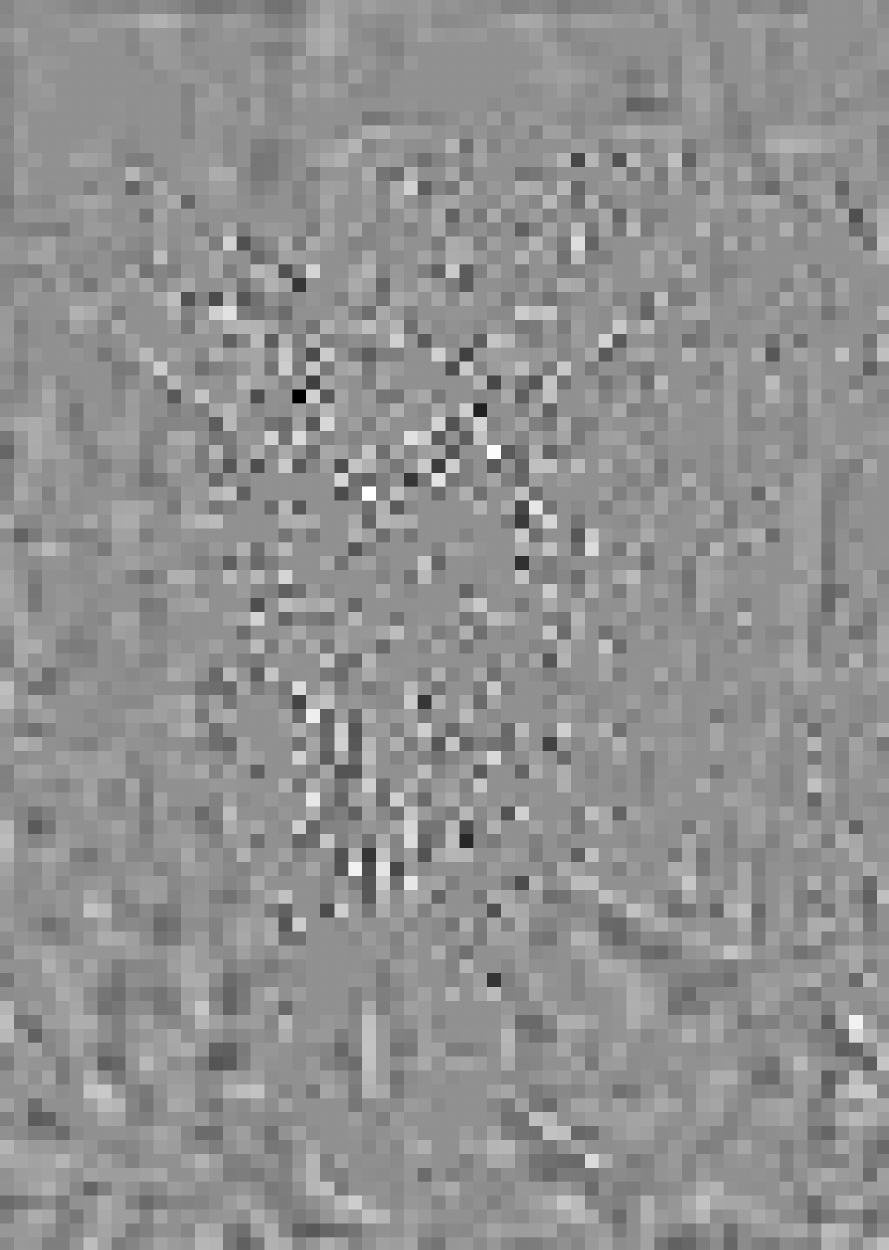}
\label{subfig:FastTV2D_F1}
}
\subfloat[$1.5\%-5\%$]
  {
\includegraphics[width=0.215\textwidth]{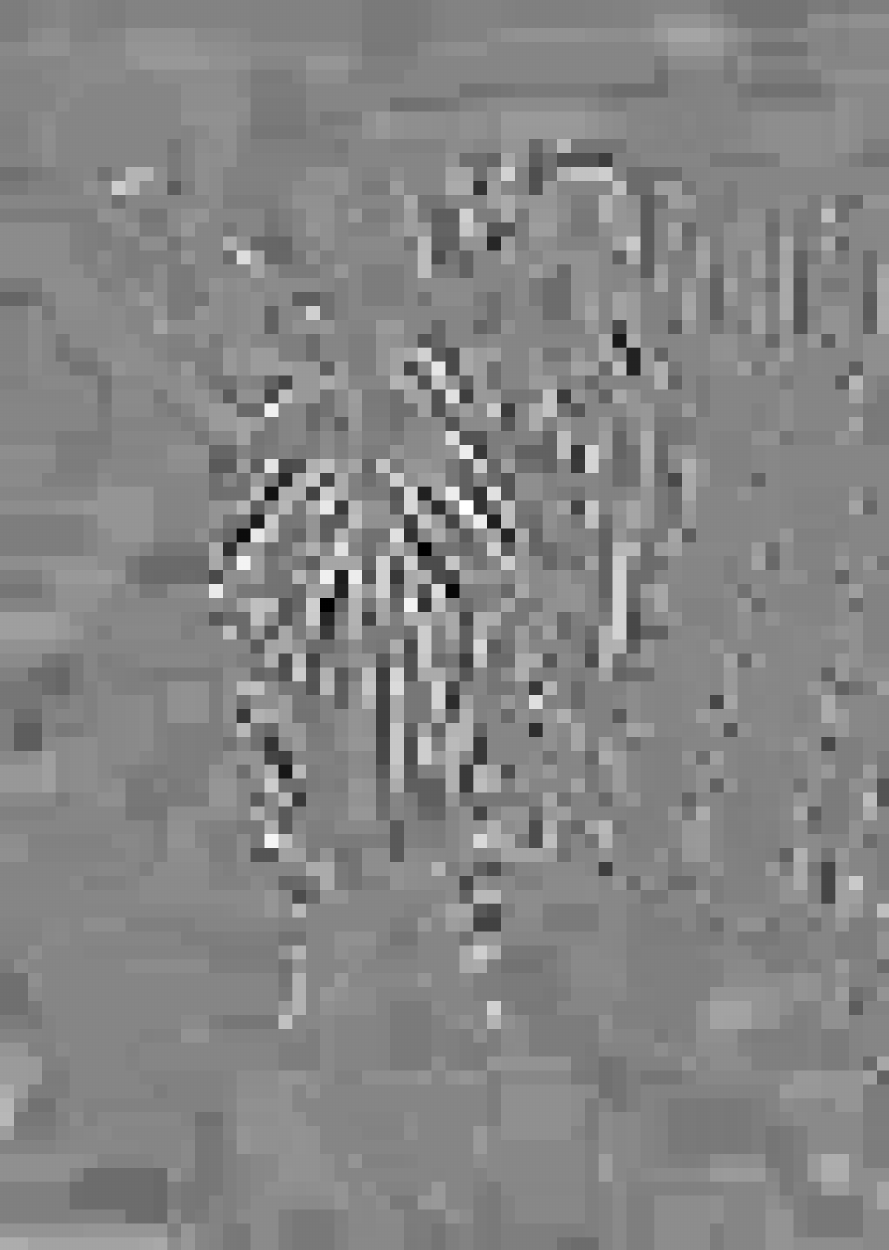}
\label{subfig:FastTV2D_F2}
}
\subfloat[$5\%-12.5\%$]
  {
\includegraphics[width=0.215\textwidth]{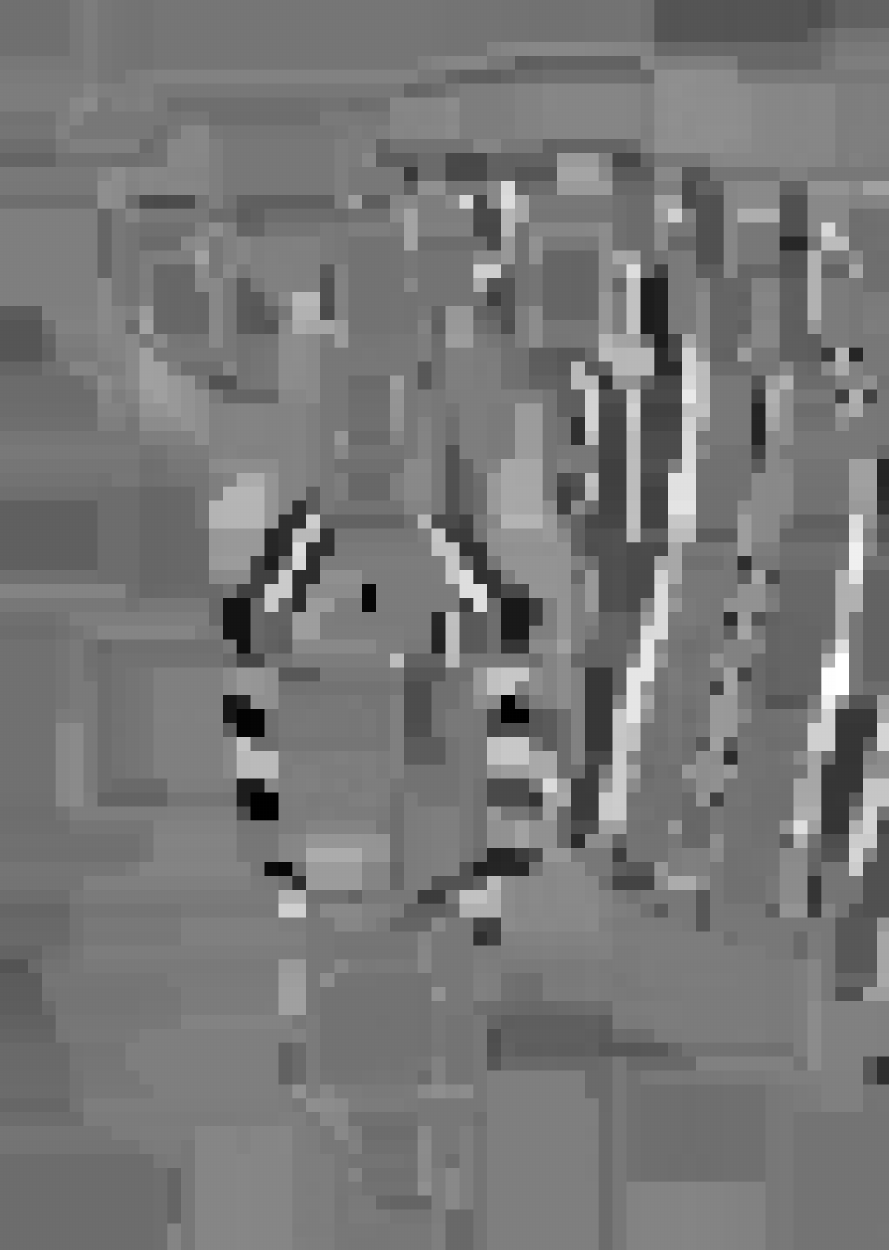}
\label{subfig:FastTV2D_F3}
}
\subfloat[$20\%-100\%$]
  {
\includegraphics[width=0.215\textwidth]{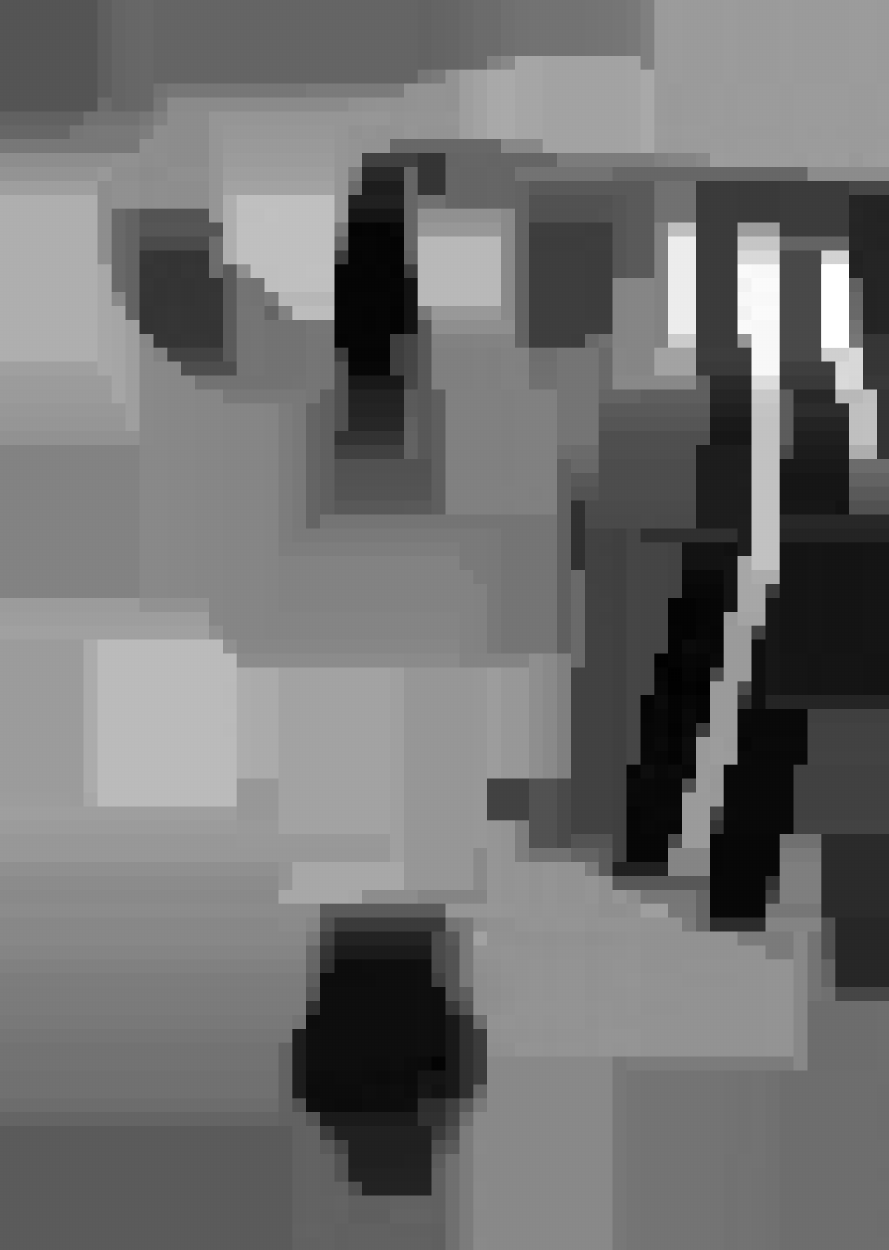}
\label{subfig:FastTV2D_F4}
}
\caption{Zebra image anisotropic \ac{TV}-decomposition, computed by the accelerated flow, Theorem \ref{theo:convergence} and Algo. \ref{algo:FastSubgradient}.}
\label{Fig:ZebraDecomposition_TVfast}
\end{figure}
We show the corresponding decay profile decomposition in Fig. \ref{Fig:ZebraDecomposition_TViso_KMD}
\begin{figure}[phtb]
\centering
\captionsetup[subfigure]{justification=centering}
\subfloat[$0\%-1.5\%$]
  {
\includegraphics[width=0.215\textwidth]{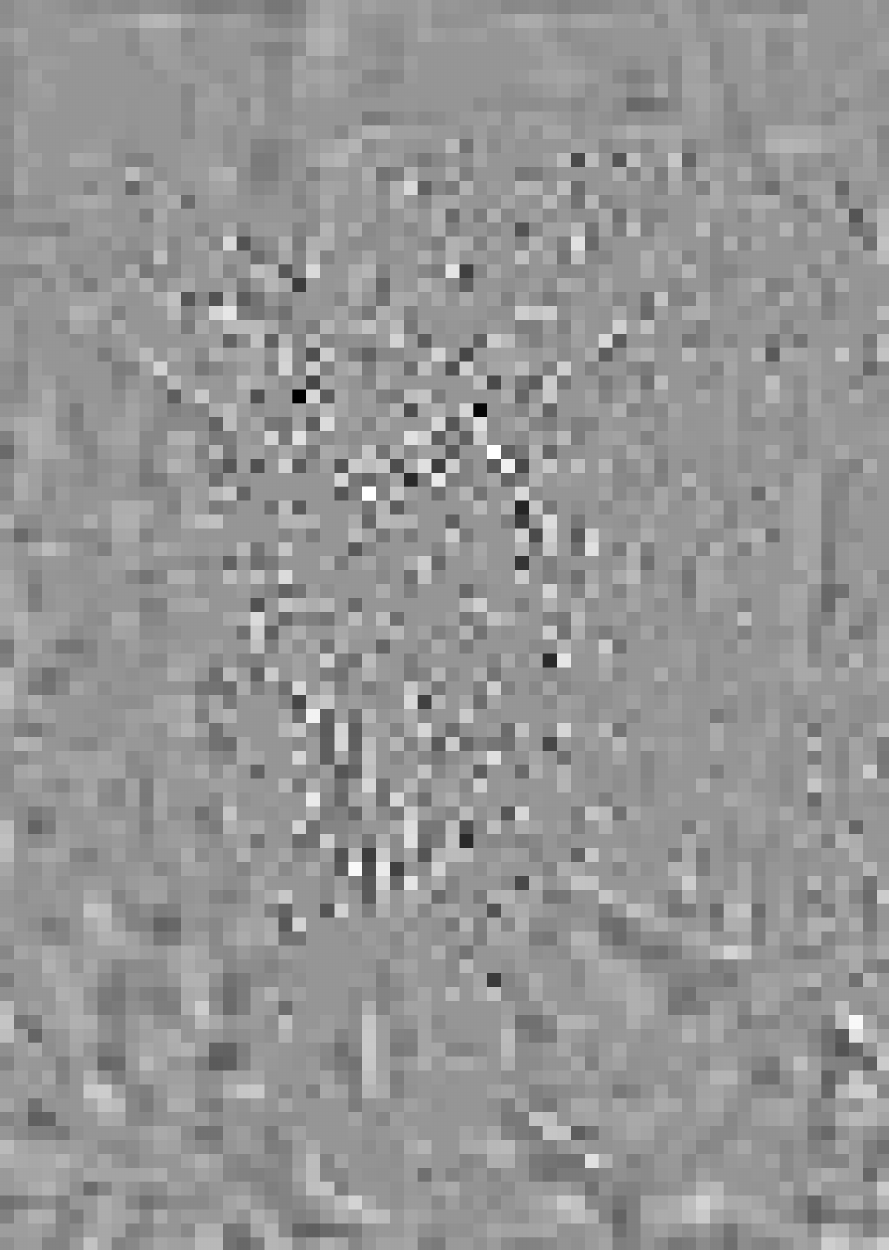}
\label{subfig:TV2Diso_KMD_F1}
}
\subfloat[$1.5\%-5\%$]
  {
\includegraphics[width=0.215\textwidth]{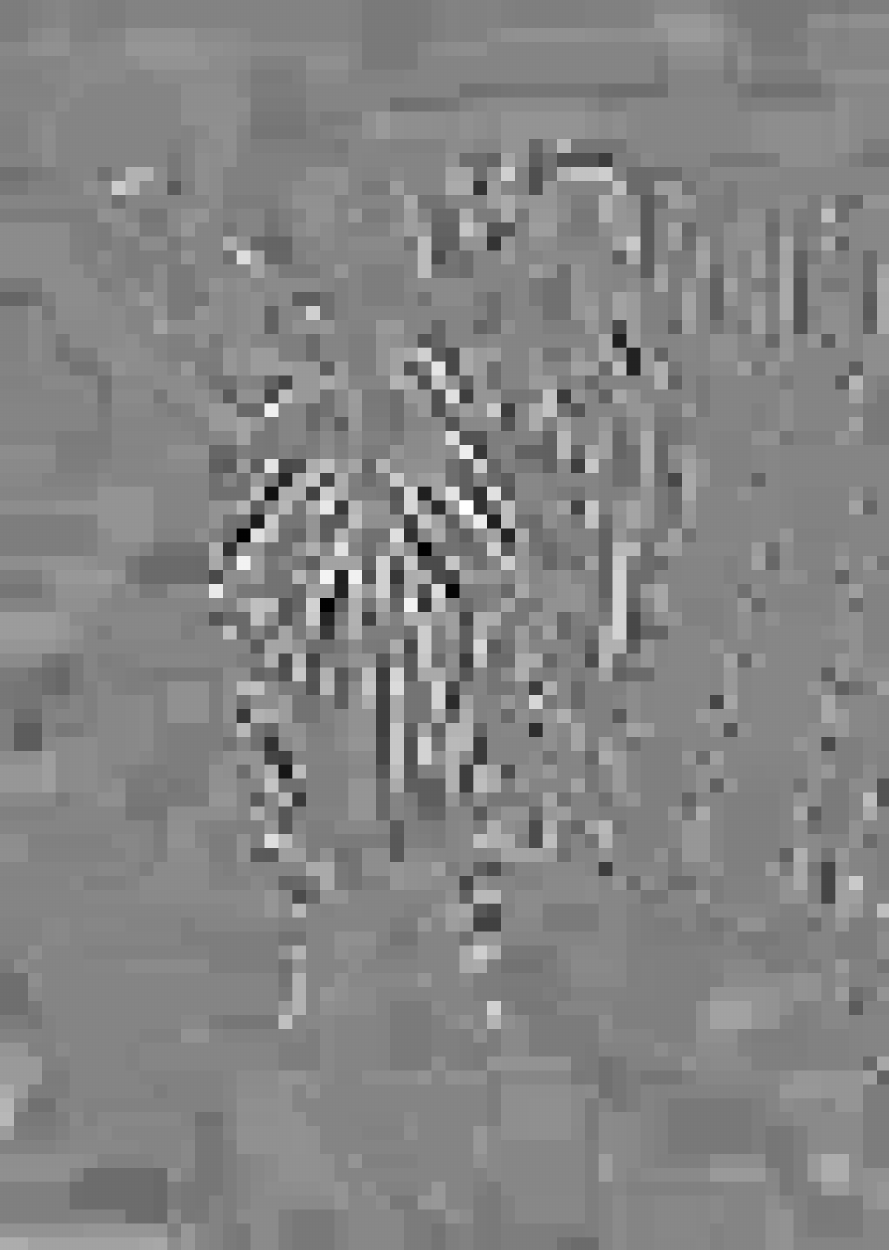}
\label{subfig:TV2Diso_KMD_F2}
}
\subfloat[$5\%-12.5\%$]
  {
\includegraphics[width=0.215\textwidth]{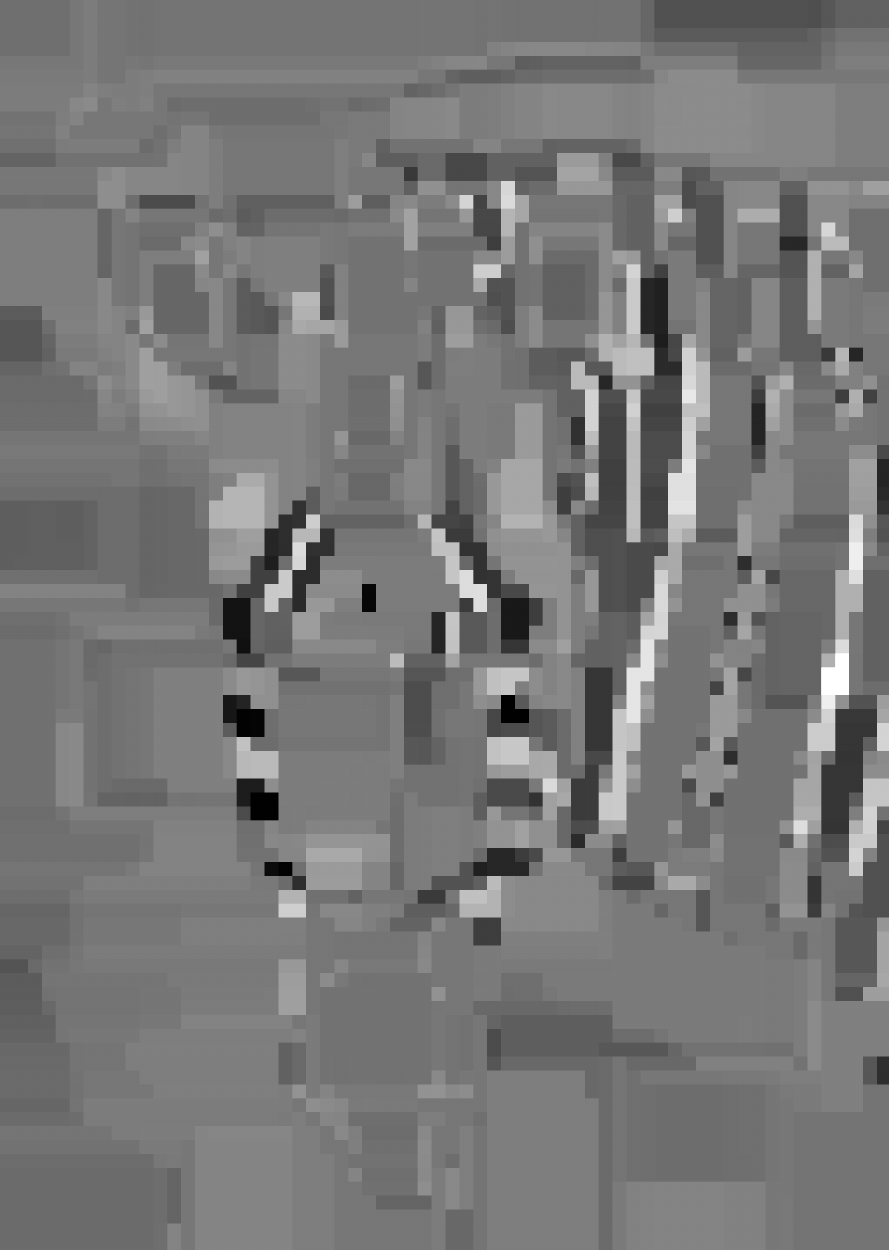}
\label{subfig:TV2Diso_KMD_F3}
}
\subfloat[$20\%-100\%$]
  {
\includegraphics[width=0.215\textwidth]{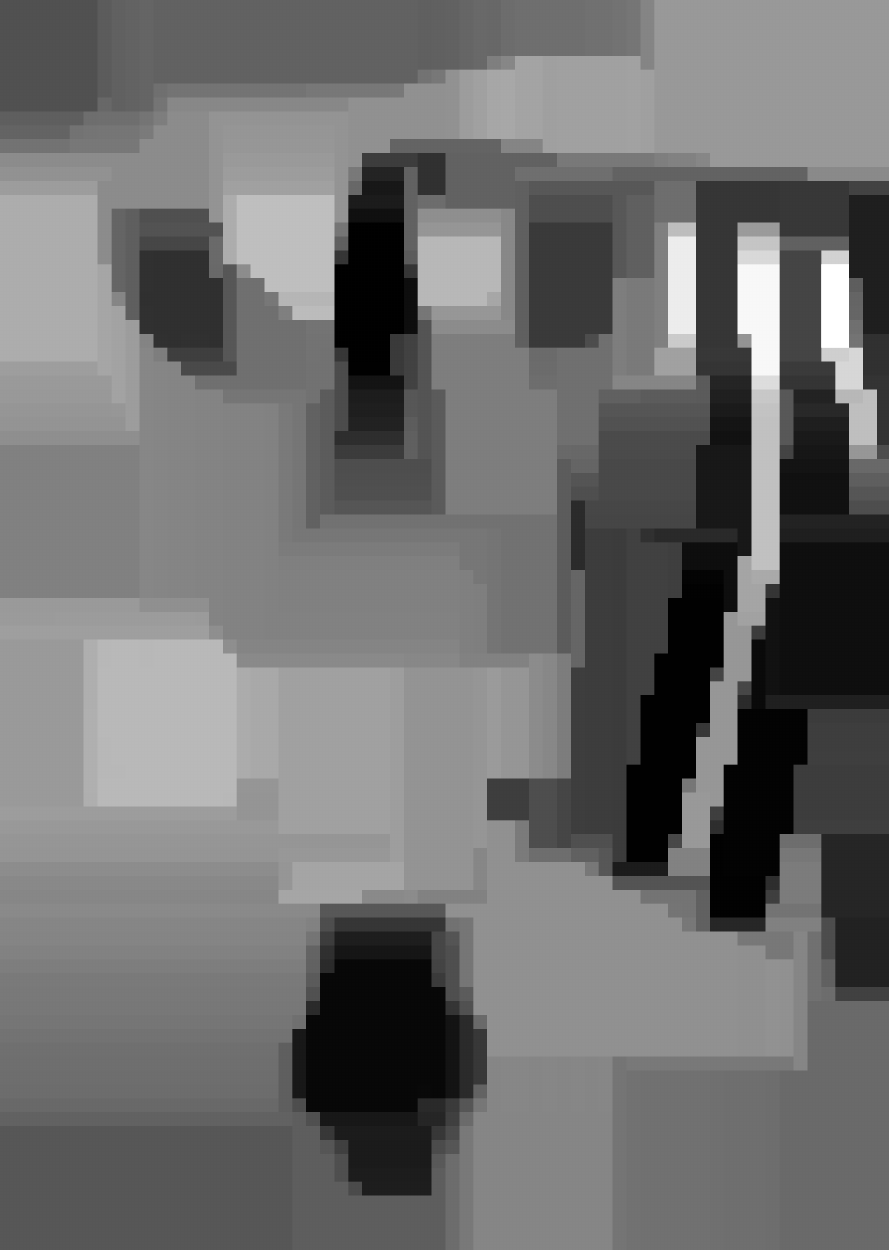}
\label{subfig:TV2Diso_KMD_F4}
}
\caption{Zebra image - decay profile decomposition with Koopman modes of anisotropic \ac{TV} flow, computed by the accelerated flow, Theorem \ref{theo:convergence} and Algo. \ref{algo:FastSubgradient}.}
\label{Fig:ZebraDecomposition_TViso_KMD}
\end{figure}

\section{Conclusion}\label{sec:con}
In this paper we thoroughly examined the \ac{DMD}\cite{schmid2010dynamic} algorithm and Koopman Theory as tools for spectral analysis and decomposition of the \ac{TV}-flow. We proposed the \acf{RDMD} adaptation as a means to overcome difficulties in \ac{DMD} application due to the linear-decay nature of the \ac{TV} flow. We have found exact relations between \ac{TV} spectral decomposition, the \acf{KMD} algorithm and \ac{DMD}. Due to the discontinuity of the dynamic, a decomposition based on the decay profile \cite{cohen2021examining} is called for. We applied this decomposition to separate the flow into Koopman modes, to compare them against the original \ac{TV} spectral decomposition. 

Since evolving \ac{TV}-flow is a slow process using optimization techniques, we have proposed a very fast method, based on simple updates of the subgradient. Finally, our accelerated algorithm was extended for solving the two-dimensional anisotropic TV-Flow.

{\bf{Acknowledgements.}}
This work was supported by the European Union’s Horizon 2020 research and innovation programme under the Marie Sk{\l}odowska-Curie grant agreement No. 777826 (NoMADS). GG acknowledges support by the Israel Science Foundation (Grant No.  534/19) and by the Ollendorff Minerva Center.


\bibliographystyle{spmpsci}      
\bibliography{smartPeople}   

\end{document}